\documentclass{amsart}

\usepackage{hyperref}

\numberwithin{equation}{section}

\makeatletter
\renewcommand{\@biblabel}[1]{#1\hfill \hspace{-0.2cm}}
\makeatother

\usepackage[frak]{paper_diening}

\usepackage[a4paper,top=1.5cm,bottom=1.5cm,left=2.5cm,right=2.5cm]{geometry}

\usepackage{amsmath}
\usepackage{amssymb, bm, bbm}
\usepackage{colortbl}
\usepackage{esint}

\newtheorem{theorem}{Theorem}[section]

\newtheorem{lemma}[theorem]{Lemma}
\newtheorem{proposition}[theorem]{Proposition}

\theoremstyle{definition}

\newtheorem{remark}{Remark}

\newcommand{\kabs}[1]{\ensuremath{\vert#1\vert}}

\newcommand{\R}{\ensuremath{\mathbb{R}}}

\renewcommand{\d}{\mathrm{d}}

\renewcommand{\Xi}{X}

\newcommand{\RN}{\setR^N}

\makeatletter
\@namedef{subjclassname@2020}{\textup{2020} Mathematics Subject Classification}
\makeatother

\begin{document}

\title[Degenerate systems of double phase type]{Partial regularity for degenerate systems of double phase type}

\author{%
Jihoon Ok, Giovanni Scilla 
   and
    Bianca Stroffolini
}




\begin{abstract} 
We study partial regularity for degenerate elliptic systems of double-phase type, where  the growth function is given by $H(x,t)=t^p+a(x)t^q$ with  $1<p\le q$ and $a(x)$ a nonnegative $C^{0,\alpha}$-continuous function. Our main result proves that if $\frac{q}{p}\le 1+\frac{\alpha}{n}$, the gradient of any weak solution is locally H\"older continuous, except on a set of measure zero.
\end{abstract}

\keywords{ double phase; degenerate system; partial regularity; harmonic approximation.}

\subjclass[2020]{
35J47, 
35J92, 
35B65, 
46E30, 
}

\maketitle


\section{Introduction} 

This article deals with nonlinear degenerate elliptic systems of double phase type: 
\begin{equation}
\quad{\rm div}\,{\bf A}(x,D\bfu)= {\bf 0}
\,\, \mbox{ in $\Omega$,}
\label{system1}
\end{equation}
where $\Omega\subset\R^n$ with $n\ge 2$ is an open set, $\bfu=(u^1,\dots,u^N)$ with $N\ge 1$, and ${\bf A}:\Omega\times\R^{N\times n}\to \R^{N\times n}$. More precisely, let  $H:\Omega\times [0,\infty)\to [0,\infty)$ be defined by \begin{equation}
H(x,t):= t^p + a(x) t^q\,,
\label{eq:H}
\end{equation}
with $a:\Omega\to \R$ and $1<p\le q$   satisfying
\begin{equation}
a\in C^{0,\alpha}(\Omega)\,, 
\quad
0\le a \le L \,,
\quad \text{and} \quad
\frac{q}{p}\leq 1+\frac{\alpha }{n}\,,
\label{eq:pq}
\end{equation}
for some $\alpha \in(0, 1]$ and $L >0$. Then the nonlinearity  
 ${\bf A}(x,\bm\xi)$ satisfies  the following double phase type growth and ellipticity conditions: 
\begin{equation}
|{\bf A}(x,\bm\xi)|+ |\bm\xi||D_\xi {\bf A}(x,\bm\xi)| \leq L H'(x, |\bm\xi|)\,, 
\tag{A1}
\label{eq:1.8ok1}
\end{equation}
\begin{equation}
\langle D_\xi {\bf A}(x,\bm\xi)\bm\lambda \, |\,\bm\lambda\rangle \geq \nu H''(x, |\bm\xi|)|\bm\lambda|^2\,,
\tag{A2}
\label{eq:1.8ok2}
\end{equation}
for every $x\in\Omega$ and $\bm\xi\in\R^{N\times n}$, $\bm\lambda \in \R^{N\times n}$ and for some $0<\nu\leq L$, where $H'(x,t)$ and $H''(x,t)$ denote the first and second derivatives of $t\to H(x,t)$, respectively, and $\langle\cdot \,|\,\cdot\rangle$ denotes the Euclidean inner product in $\R^{N\times n}$.
Note that in the region where $a(x)=0$, we have $H(x,t)=t^p$ so that $H$ has a $p$-phase, while in the region where $a(x)>0$, the function $H$ has a  $(p,q)$-phase.  In particular, if $a(x)\equiv 0$, then $H(x,t)\equiv t^p$ and thus $A(x,\xi)$ satisfies the standard $p$-growth condition.  We also note that \eqref{eq:1.8ok2} implies that for every $x\in\Omega$ and $\bm\xi_1, \bm\xi_2\in \R^{N\times n}$,
\begin{equation}
\langle {\bf A}(x,\bm\xi_1) - {\bf A}(x,\bm\xi_2) \, | \, \bm\xi_1-\bm\xi_2 \rangle \geq \tilde{\nu} H''(x, |\bm\xi_1|+|\bm\xi_2|)|\bm\xi_1-\bm\xi_2|^2
\label{eq:1.9ok}
\end{equation}
for some $\tilde \nu>0$ depending on $p$, $q$ and $\nu$.
We say that a function $\bfu \in W^{1,1}(\Omega;\R^N)$ with $H(\cdot,|D\bfu|)\in L^1(\Omega)$ is a \emph{weak solution} to \eqref{system1} if 
\begin{equation}
\int_\Omega \langle {\bf A}(x,D\bfu)\, |\,  D\bm \psi\rangle\,\mathrm{d}x = 0
\label{eq:system1}
\end{equation}
holds for all $\bm \psi\in W^{1,1}_0(\Omega;\R^N)$ with $H(\cdot,|D\bm \psi|)\in L^1(\Omega)$.

The H\"older continuity of the gradient of weak solutions for degenerate or singular elliptic systems of the form \eqref{system1} has been a long-standing and still active research topic. In the case of a single equation, i.e., $N=1$, if $H(x,t)\equiv H_0(t):=t^p+a_0 t^q$ for some constant $a_0\ge 0$, and $A(x,\xi)$ satisfies conditions \eqref{eq:1.8ok1}, \eqref{eq:1.8ok2} along with a suitable H\"older continuous condition for the $x$ variable,
 then the gradient of the weak solution to \eqref{system1} is locally H\"older continuous. See \cite{Lie}, and also \cite{Ural,Manfredi} for the case  $a_0=0$. (Note that the paper \cite{Lie} considers a more general function than $H_0(t)$.) However, this everywhere regularity result does not extend to the vectorial case, particularly when $N\ge 2$ and $n\ge 3$, even if $A(x,\xi)\equiv A_0(\xi)$ and $H(x,t)=t^2$. See \cite{mingionedark} for more discussion for vectorial problems.   

Neverthelss, if the system satisfies an isotropic structure, the gradient of the weak solution is locally H\"older continuous. Uhlenbeck \cite{Uhlenbeck} proved that  if $A(x,\bm\xi)\equiv \varrho(|\bm\xi|^2)\bm\xi$ and $\varrho(t)\sim |t|^{\frac{p-2}{2}}$ with $p>2$ ( see \cite[(1.3) and (1.4)]{Uhlenbeck} for detailed conditions on $\varrho$ which are known as the Uhlenbeck structure condition), then the gradient of the weak solution is  locally H\"older continuous. The prototype of a system satisfying the Uhlenbeck structure is the $p$-Laplace system
$$
\mathrm{div} \left(|D\bfu|^{p-2} D\bfu\right)= {\bf 0}
\quad \text{in }\ \Omega.
$$
This result was extended to the case $1<p<2$ by Tolksdorf \cite{Tolksdorf}. We remark that the $p$-Laplace system is degenerate when $p>2$, and singular when $1<p<2$, since $|D\bfu|^{p-2}$ tends to $\infty$ or $0$ as $|D\bfu|\to 0$, respectively.  Moreover, by setting $\varrho(t)=\phi'(\sqrt{t})/\sqrt{t}$, where $\phi'$ is the derivative of a given function $\phi$, we have 
$$
\mathrm{div} \left(\frac{\phi'(|D\bfu|)}{|D\bfu|} D\bfu \right)= {\bf 0}
\quad\text{in }\ \Omega,
$$
which is the Euler-Lagrange system of the isotropic energy $\int_\Omega\phi(|D\bfu|)\, \d x$.
For this system, Diening, Stroffolini and Verde \cite{DIESTROVER09}  proved  everywhere $C^{1,\alpha}$-regularity  by assuming a suitable condition on $\phi$, which generalizes the Uhlenbeck condition \cite[(1.3) and (1.4)]{Uhlenbeck}.
 
Although we cannot expect everywhere H\"older continuity of the gradient of weak solutions to general degenerate system \eqref{system1}, partial H\"older continuity, that is, H\"older continuity except on a Lebesgue measure zero set, can still be achieved by assuming a suitable condition on $A(x,\bm\xi)$ additionally. Duzaar and Mingione \cite{DUMIN04b}  first obtained partial H\"older continuity results for the gradient of weak solutions when $A(x,\bm\xi)=A_0(\bm\xi)$ and $H(x,t)\equiv t^p$. The key tool in  their proof is harmonic approximation: specifically, $\mathcal A$-harmonic and $p$-harmonic approximation. The $\mathcal A$-harmonic approximation was first used in partial regularity theory by Duzaar and Grotwoski \cite{DUGRO00}; see also \cite{DUST02}. On the other hand, the $p$-harmonic approximation was first obtained in \cite{DUMIN04}. Note that the harmonic approximation results in \cite{DUGRO00,DUMIN04} are proved by using contradiction argument. See \cite{DUMIN09} for more discussion on harmonic approximation. Later, harmonic approximations in more general settings have been obtained by Diening, Lengeler, Stroffolini and Verde \cite{DIELENSTROVER12,DIESTROVER10},  which will be introduced in Section~\ref{sec:harmonicapproximation},
where the proofs employ Lipschitz truncation argument instead of  the contradiction method.
For further results on partial regularity in degenerate systems or relevant variational problems, we refer to \cite{BeckSt13,Bo12,BoDuMin13,CeladaOk,GoodScSt22,OkScSt24}.

The  energy of the form
\begin{equation}\label{energy}
\int_\Omega|D\bfu|^p + a(x)|D\bfu|^q\,\d x\,,
\end{equation}
known as the double phase energy, and related equations and systems have been intensively studied  over the last decade, especially following sharp and comprehensive regularity results obtained in the papers of Baroni, Colombo and Mingione \cite{ColomboMingione15,ColomboMingione151,BCM15,BCM18}. The double phase energy has been first introduced by Zhikov \cite{Zhikov86,Zhikov95,Zhikov97} in the context of Homogenization and as an example exhibiting Lavrentiev phenomenon. Subsequently, Esposito, Leonetti and Mingione \cite{ELM04}, as well as Fonseca, Mal\'y and Mingione \cite{FMM04}, provided further examples of double phase problems related to Lavrentiev phenomenon, which also establish the sharpness of the condition \eqref{eq:pq} in order to obtain regularity results. In particular, the latter paper also investigates some regularity results namely, higher integrability and fractional differentiability of corresponding solutions. Finally, a sharp and maximal regularity result has been obtained in \cite{ColomboMingione15,ColomboMingione151,BCM18}. Specifically, it was shown that if $\bfu$ is a minimizer of the energy \eqref{energy} with the condition \eqref{eq:pq}, then $D\bfu\in C^{0,\gamma}_{\loc}(\Omega)$ for some $\gamma\in(0,1)$.  For further regularity results related to double phase type problems, we refer to \cite{BBK23,ByunOh17,BOS22,ColomboMingione16,cristianaMin20,cristianaOh19,cristianaPalatucci19,KKM23,OK2017,OKJFA18,OKNLA18,OKNLA20,SciStroffofluids} and the references therein.

Partial regularity for nondegenerate systems  with double phase growth has been studied in \cite{OKJFA18,OKNLA18,SciStroffofluids}, where the term `nondegenerate' for the system \eqref{system1} means  $|D_{\bm \xi}{\bf A}(x,{\bf 0})|\sim 1$. Furthermore, those papers assume the superquadratic condition, namely the smaller exponent $p$ of $H$ is larger that or equal to $2$. However, the general case of degenerate systems with double phase growth has not yet been explored.  In this paper, we consider degenerate systems of the form \eqref{system1} with double phase growth and prove partial H\"older regularity of the gradient of their weak solutions (see Theorem~\ref{thm:1:main-thm}). Notably, we do not assume that the superquadratic condition holds, and  thus a unified approach independent of the  the exponent $p$ is required.

The proof of the main theorem (Theorem~\ref{thm:1:main-thm}) is based on the harmonic approximation approach introduced in \cite{DUMIN04b}, which has been further developed for the Orlicz growth case in \cite{CeladaOk,GoodScSt22}. The double phase growth condition presents different phenomena, as it does not imply uniform ellipticity with respect to the gradient variable. Therefore, we need to develop the approach in our setting.   
Moreover, the excess functional and methodology used in \cite{OKJFA18,OKNLA18} are not working on this paper, since 
we are also dealing with the degenerate case as well as both the superquadratic ($p\ge 2$) and the subquadratic ($1<p<2$) cases at the same time. 
To handle this challenge, we introduce the following new excess functional:
\begin{equation*}
\Phi (x_0,r, {\bf Q}):  = \frac{1}{H^-_{B_r(x_0)}(|{\bf Q}|)}\dashint_{B_{r}(x_0)} |\bfV_{H^-_{B_{r}(x_0)}}(D\bfu)-\bfV_{H^-_{B_{r}(x_0)}}({\bf Q})|^2 \,\mathrm{d} x\,,
\end{equation*} 
where
\begin{equation*}
\bfV_{H^-_{B_r(x_0)}}({\bf Q}) =  \sqrt{\frac{(H^-_{B_r(x_0)})'(|{\bf Q}|)}{|{\bf Q}|}} {\bf Q}\,, \quad {\bf Q}\in\R^{N\times n}\,, \quad \mbox{ and }\,\, \,\, H^-_{B_r(x_0)}(t):= t^p + \inf_{x\in B_r(x_0)}a(x) t^q\,,\,\, t\geq0\,.
\end{equation*}
Here, the double phase function $H$ is frozen at the infimum of the modulating coefficient $a(x)$ on the ball $B_r(x_0)$.

We distinguish between degenerate and nondegenerate regimes using an alternative condition, and prove that weak solutions to \eqref{system1}, or its variations, are almost  $\phi$-harmonic with $\phi(t)=t^p+at^q$ for some $a\ge 0$, or almost $\mathcal A$-harmonic for some $\mathcal A\in R^{N^2\times n^2}$ that satisfies the Legendre-Hadamard ellipticity condition. Note that we establish almost harmonicity through a new approach which is applicable to the $p$-phase and the $(p,q)$-phase at the same time. This potentially 
simplifies known comparison arguments such as those in \cite{ColomboMingione15, ColomboMingione16, BCM18}).  
Finally we apply $\mathcal A$-harmonic and $\phi$-harmonic approximation lemmas to derive excess decay estimates. 
Furthermore, we emphasize that our new excess functional depends on the radius of the ball $B_r(x_0)$. This requires a more delicate analysis in the iteration process.

The remainder of the paper is organized as follows. In Section~\ref{sec:mainresult}, we present our main result, Theorem~\ref{thm:1:main-thm}. Section~\ref{sec:prelim} introduces the necessary notation, elementary inequalities, harmonic functions and harmonic approximation results. In section~\ref{sec:caccio_reverse}, we obtain Caccioppoli type estimates and reverse H\"older type inequalities with higher integrability estimates.  Section~\ref{sec:decayestimate} provides decay estimates for both nondegenerate and degenerate regimes and iterate those decay estimates on shrinking balls.  Finally, in Section~\ref{sec:proofmainthm}, we prove our main result Theorem~\ref{thm:1:main-thm}.

\subsection{Main result}\label{sec:mainresult}

Let us introduce the main result of the paper.
We assume additional conditions on the nonlinearity ${\bf A}(x,\bm\xi)$, in order to prove a partial regularity result. 
For the dependence on the $x$-variable, we assume that 
\begin{equation}
|{\bf A}(x_1,\bm\xi)-{\bf A}(x_2,\bm\xi)| \leq  L |x_1-x_2|^{\beta_0}\left(H'(x_1,|\bm\xi|)+ H'(x_2,|\bm\xi|)\right)  + L |a(x_1)- a(x_2)| |\bm\xi|^{q-1}\,,
\tag{A3}
\label{eq:1.11ok}
\end{equation}
for some $\beta_0\in(0,1)$, for every $x_1,x_2\in\Omega$ and every $\bm\xi\in \R^{N\times n}$. 
For the non-degenerate case, we further require
\begin{equation}
|D_\xi{\bf A}(x,\bm\xi_1)-D_\xi{\bf A}(x,\bm\xi_2 + \bm\xi_1)| \leq L \left(\frac{|\bm\xi_2|}{|\bm\xi_1|}\right)^{\beta_0} H''(x,|\bm\xi_1|)\,, \quad\mbox{whenever }\  |\bm\xi_2| \leq \frac{1}{2}|\bm\xi_1|\,,
\label{eq:1.12ok}
\tag{A4}
\end{equation}
for every $x\in\Omega$ and $\bm\xi_1, \bm\xi_2\in \R^{N\times n}$.
For the degenerate case, instead, we assume the following: for every $\delta>0$, there exists $\kappa=\kappa(\delta)>0$ such that
\begin{equation}
\left |{\bf A}(x,\bm\xi)-H'(x,|\bm\xi|)\frac{\bm\xi}{|\bm\xi|}\right| \leq \delta H'(x,|\bm\xi|)\,,
\quad 
\text{whenever }\ 0<|\bm\xi|\leq \kappa\,,
\label{eq:degenereassump}
\tag{A5}
\end{equation}
for every $x\in\Omega$ and $\bm\xi\in \R^{N\times n}$. 

\begin{theorem}\label{thm:1:main-thm}
Let $H:\Omega\times[0,\infty)\to[0,\infty)$ be defined as in \eqref{eq:H} complying with \eqref{eq:pq}, and  ${\bf A}:\Omega\times \R^{N\times n}\to \R^{N\times n}$ comply with \eqref{eq:1.8ok1}--\eqref{eq:degenereassump}.
If $\bfu \in W^{1,1}(\Omega;\R^N)$ with $H(\cdot,|D\bfu|)\in L^1(\Omega)$ is a weak solution to \eqref{system1}, then there exist $\beta=\beta(n,N,p,q,\nu,L,\alpha, [a]_{C^{0,\alpha}},\beta_0)\in(0,1)$ and an open subset $\Omega_0 \subset \Omega$ such that $\kabs{\Omega \setminus \Omega_0}  = 0$ and
\begin{equation*}
 D\bfu \in C^{0, \beta}_{\rm loc}\left(\Omega_0 ;\R^{N\times n}\right)  \,.
\end{equation*}
Moreover, $\Omega \setminus \Omega_0\subset \Sigma_1\cup\Sigma_2$ where
\begin{equation*}
\begin{split}
&\Sigma_1:=\left\{x_0\in\Omega:\,\, \liminf_{r \to 0^+} \dashint_{B_r(x_0)} \vert \bfV_{H^-_{B_r(x_0)}}(D\bfu) - (\bfV_{H^-_{B_r(x_0)}}(D\bfu))_{x_0,r} \vert^2 \, \mathrm{d}x>0\right\}\,,\\
&\Sigma_2:=\left\{x_0\in\Omega:\,\, \limsup_{r \to 0^+} \dashint_{B_r(x_0)} \vert  \bfV_{H^-_{B_r(x_0)}}(D\bfu) \vert^2 \, \mathrm{d}x  = \infty\right\}\,.
\end{split}
\end{equation*} 
where $H^-_{B_r(x_0)}(t)$ and  $\bfV_{H^-_{B_r(x_0)}}({\bf P})$ are defined as in \eqref{eq:Hpmdef}  and \eqref{eq:defV} with $\varphi(t)=H^-_{B_r(x_0)}(t)$, respectively.
\end{theorem}

\begin{remark}
We note from \eqref{eq:Hpmdef}, \eqref{eq:linkVphi} and \eqref{eq:equivalencebis} that
$$
\dashint_{B_r(x_0)} \vert  \bfV_{H^-_{B_r(x_0)}}(D\bfu) \vert^2 \, \mathrm{d}x  \le \dashint_{B_r(x_0)} \vert  \bfV_{H}(D\bfu) \vert^2 \, \mathrm{d}x  
$$
and
 $$
 \dashint_{B_r(x_0)} \vert \bfV_{H^-_{B_r(x_0)}}(D\bfu) - (\bfV_{H^-_{B_r(x_0)}}(D\bfu))_{x_0,r} \vert^2 \, \mathrm{d}x \lesssim \dashint_{B_r(x_0)} \vert \bfV_{H}(D\bfu) - (\bfV_{H}(D\bfu))_{x_0,r} \vert^2 \, \mathrm{d}x \,.
 $$ 
 Therefore, since $|\bfV_H(D\bfu)|^2 \sim  H(\cdot,|D\bfu|) \in L^1(\Omega)$, by Lebesgue differentiation theorem, we see that $|\Sigma_1|=|\Sigma_2|=0$.
 \end{remark}

\section{Preliminaries and auxiliary results} \label{sec:prelim}

\subsection{Basic notation}\label{sec:notation}
We denote by $\Omega$ an open bounded domain of $\R^n$. For $x_0\in\R^n$ and $r>0$, $B_r(x_0)$ is the open ball of radius $r$ centred at $x_0$. In the case $x_0=0$, we will often use the shorthand $B_r$ in place of $B_r(x_0)$. If $f\in L^1(B_r(x_0); \R^m)$, we denote the average of $f$ by
\begin{equation*}
(f)_{x_0,r}:= \dashint_{B_r(x_0)} f\,\mathrm{d}x\,.
\end{equation*}
We denote by $\R^{N\times n}$ the set of all $N\times n$ matrices. For ${\bm a}=(a^1,\dots, a^N)\in \R^N $ and $x=(x_1,\dots,x_n)\in\R^n$, we denote their tensor product by ${\bm a}\otimes x:=\{a^ix_j\}_{i,j}\in \R^{N\times n}$. For a function $\bfu \in L^1(\Omega;\R^N)$, we denote by $D\bfu$ its distributional derivative in $\R^{N\times n}$.
If $p>1$, then $p':=\frac{p}{p-1}$ denotes the H\"older conjugate exponent of $p$. If $1<p<n$, the number $p^*:=\frac{np}{n-p}$ stands for the Sobolev conjugate exponent of $p$, whereas $p^*$ is any real number larger than $p$ if $p\geq n$. $\mathbbm{1}_{U}$ is the characteristic function with respect to $U\subset \R^n$, that is, $\mathbbm{1}_{U}(x)=1$ if $x\in U$ and $\mathbbm{1}_{U}(x)=0$ if $x\not\in U$. $f \lesssim g$ means $f\le c g$ for some $c\ge 1$ depending on structure constants, and $f\sim g$ means    $f\lesssim g$ and  $g \lesssim f$.

\subsection{Some basic facts on $N$--functions} \label{sec:basicNfunctions}

We recall here some elementary definitions and basic results about Orlicz functions. The following definitions and results can be found, e.g., in \cite{Kras, Kufn, Bennett, Adams}. 


A real-valued function $\phi:[0,\infty)\to[0,\infty)$ is said to be an \emph{$N$-function} if it is convex and satisfies the following conditions: $\phi(0)=0$, $\phi$ admits the derivative $\phi'$ and this derivative is right continuous, non-decreasing and satisfies $\phi'(0) = 0$, $\phi'(t)>0$ for $t>0$, and $\lim_{t\to \infty} \phi'(t)=\infty$. 
We say that $\phi$ satisfies the \emph{$\Delta_2$-condition} if there exists $c > 0$ such that for all $t \geq 0$ holds $\phi(2t) \leq c\,\phi(t)$. 
We denote the smallest possible such constant by $\Delta_2(\phi)$. 
Since $\phi(t) \leq \phi(2t)$, the $\Delta_2$-condition is equivalent to $\phi(2t) \sim \phi(t)$.

For an $N$-function $\phi$, we assume that
\begin{equation}
p_1\leq \inf_{t>0}\frac{t\phi'(t)}{\phi(t)}\leq \sup_{t>0}\frac{t\phi'(t)}{\phi(t)}\leq p_2\,,
\label{(2.1celok)}
\end{equation}
for some $1<p_1\leq p_2 <\infty$. Furthermore, we can also assume that $\phi\in C^2((0,\infty))$ satisfies 
\begin{equation}
0<p_1-1\leq \inf_{t>0}\frac{t\phi''(t)}{\phi'(t)}\leq \sup_{t>0}\frac{t\phi''(t)}{\phi'(t)}\leq p_2-1\,.
\label{(2.1celokbis)}
\end{equation}
Note that if  $\phi$ satisfies \eqref{(2.1celokbis)}, then \eqref{(2.1celok)} holds hence we have
\begin{equation}
\phi(t) \sim t \phi'(t)
\quad\text{and}\quad
 \phi(t) \sim t^2 \phi''(t)\,, \quad t>0\, .
\label{ineq:phiast_phi_p}
\end{equation}
For instance, $\phi(t):=t^p$, $1<p<\infty$, is an $N$-function satisfying \eqref{(2.1celokbis)} with $p_1=p_2=p$.  Also, for  $H$ defined as in \eqref{eq:H} and each $x\in \Omega$, $\phi(t):=H(x,t)$ is an $N$-function satisfying \eqref{(2.1celokbis)} with $p_1=p$ and $p_2=q$. 

 We denote the Young-Fenchel-Yosida conjugate function of $\phi$ by $\phi^*(t):= \sup_{s \geq 0} (st - \phi(s))$. It is again an $N$-function; it satisfies \eqref{(2.1celok)} with $\frac{p_2}{p_2-1}$ and $\frac{p_1}{p_1-1}$ in place of $p_1$ and $p_2$, respectively. We will denote by $\Delta_2({\phi, \phi^\ast})$ constants depending on $\Delta_2(\phi)$ and $\Delta_2(\phi^*)$. Also, it is easy to check that $(\phi^\ast)^\ast = \phi$.

\begin{proposition}\label{prop:properties}
Let $\phi:[0,\infty)\to[0,\infty)$ be an $N$-function complying with \eqref{(2.1celok)}. Then 
\begin{itemize}
\item[$(a)$] the mappings
\begin{equation*}
t\in(0,\infty)\to \frac{\phi'(t)}{t^{p_1-1}}\,,\,\, \frac{\phi(t)}{t^{p_1}} \mbox{ \,\, and \,\, } t\in(0,\infty)\to \frac{\phi'(t)}{t^{p_2-1}}\,,\,\, \frac{\phi(t)}{t^{p_2}}
\end{equation*}
are nondecreasing and nonincreasing, respectively. In particular,
\begin{equation}
\begin{split}
\phi(at) & \leq a^{p_1}\phi(t)\,,\,\,\,\,\,\, \phi'(at) \leq a^{p_1-1}\phi'(t)\,, \quad 0<a<1\,, \\
\phi(bt) & \leq b^{p_2}\phi(t)\,,\,\,\,\,\,\, \phi'(bt) \leq b^{p_2-1}\phi'(t)\,,\quad b>1\,.
\end{split}
\label{eq:2.2ok}
\end{equation}
Moreover,
$$
\phi^*(at) \leq a^\frac{p_2}{p_2-1}\phi(t)\,,\quad \phi^*(bt) \leq a^\frac{p_1}{p_1-1}\phi(t) \,.
$$
\item[$(b)$] (Young's inequality) for any $\lambda\in(0,1]$ it holds that
\begin{equation}
\begin{split}
st & \leq \lambda^{-p_2+1}\phi(s) + \lambda \phi^*(t)\,, \\
st & \leq \lambda \phi(s) + \lambda^{-\frac{1}{p_1-1}} \phi^*(t)\,.
\end{split}
\label{eq:2.5ok}
\end{equation}
\item[$(c)$] there exists a constant $c=c(p_1,p_2)>1$ such that
\begin{equation}
c^{-1}\phi(t) \leq \phi^*(t^{-1}\phi(t)) \leq c \phi(t)\,.
\label{eq:2.6ok}
\end{equation}
\end{itemize}
\end{proposition}

In particular, it follows from \eqref{eq:2.2ok} that both $\phi$ and $\phi^*$  satisfy the $\Delta_2$-condition with constants $\Delta_2(\phi)$ and $\Delta_2(\phi^*)$ determined by $p_1$ and $p_2$.

We will use the following lemma.
(see \cite[Lemma~20]{DIEETT08})
\begin{lemma}
\label{technisch-mu}
Let $\phi$ be a $N$-function with $\phi, \phi^*\in \Delta_2$. Then for all ${\bf{P}}_0,{\bf{P}}_1\in \mathbb{R}^{N\times n}$ with $|{\bf{P}}_0|+|{\bf{P}}_1|>0$, there holds
\begin{equation*}
 \int_0^1 \frac{\phi'(|(1-\theta){\bf{P}}_0+\theta{\bf{P}}_1|)}{|(1-\theta){\bf{P}}_0+\theta{\bf{P}}_1|}\, \mathrm{d}\theta \sim \frac{\phi'(|{\bf{P}}_0|+|{\bf{P}}_1|)}{|{\bf{P}}_0|+|{\bf{P}}_1|}\,,
\end{equation*}
where  hidden constants depend only on $\Delta_2(\phi, \phi^*).$
\end{lemma}


By $L^\phi(\Omega;\R^N)$ and $W^{1,\phi}(\Omega;\R^N)$ we denote the classical Orlicz and Orlicz--Sobolev spaces, i.\,e., $f \in L^\phi(\Omega;\R^N)$ iff $f \in L^{1}(\Omega; \R^N)$ with $\int_\Omega
\phi(|{f}|)\,dx < \infty$ and $f \in W^{1,\phi}(\Omega; \R^N)$ iff $f \in W^{1,1}(\Omega; \R^N)$  with $|f|, |D f| \in L^\phi(\Omega)$. The space $W^{1,\phi}_0(\Omega;\R^N)$ will denote the closure of $C^\infty_0(\Omega;\R^N)$ in $W^{1,\phi}(\Omega;\R^N)$. For simplicity,  we write $\|f\|_{L^\phi(\Omega)}= \||f|\|_{L^\phi(\Omega)}$ for $f\in L^\phi(\Omega;\R^{N\times n})$.

%


Another important toolset is the \emph{shifted} $N$-functions
$\{\phi_a \}_{a \ge 0}$ (see \cite{DIEETT08}). We define for $t\geq0$
\begin{equation}
  \label{eq:phi_shifted}
  \phi_a(t):= \int _0^t \phi_a'(s)\, \mathrm{d}s\qquad\text{with }\quad
  \phi'_a(t):=\phi'(a+t)\frac {t}{a+t}.
\end{equation}
Note that $\phi_a$ satisfy the $\Delta_2$-condition uniformly in $a \ge 0$, and we have the following relations:
\begin{align}
&\phi_a(t) \sim \phi'_a(t)\,t\,;  \label{(2.6a)} \\
&\phi_a(t) \sim \phi''(a+t)t^2\sim\frac{\phi(a+t)}{(a+t)^2}t^2\sim \frac{\phi'(a+t)}{a+t}t^2\,,\label{(2.6b)}\\
& \phi(a+t)\sim [\phi_a(t)+\phi(a)]\,,\label{(2.6c)} 
\end{align}
We recall also that, by virtue of \cite[Lemma~30]{DIEETT08}, uniformly in $\lambda\in[0,1]$ and $a\geq0$ holds
\begin{equation}
\phi_a^*(\lambda \phi'(a)) \sim \lambda^2 \phi(a) \,.
\label{eq:6.23dieett}
\end{equation}
  
%


We define
\begin{equation}\label{eq:defV}
\bfV_{\phi} (\bfP) := \sqrt{\frac{\phi'(|\bfP|)}{|\bfP|}} \bfP, \qquad \bfP\in \R^{N\times n}.
\end{equation}
In particular, we write $\bfV_{p} (\bfP)$ to denote $\bfV_{\phi} (\bfP)$ when  $\phi(t)=t^p$. We have that
\begin{equation}
|\bfV_{\phi} (\bfP) - \bfV_{\phi} (\bfQ)|^2 \sim \phi_{|\bfQ|}(|\bfP-\bfQ|)
\label{eq:linkVphi}
\end{equation}
holds uniformly for every $\bfP, \bfQ\in \R^{N\times n}$ (see \cite[Lemma 7]{DIELENSTROVER12}). We also recall from \cite[Lemma~A.2]{DiKaSch} that for $\bfg\in W^{1,\phi}(B_r(x_0);\R^m)$,
\begin{equation}
\dashint_{B_r(x_0)}|{\bf V}_\varphi(\bfg)-{\bf V}_\varphi((\bfg)_{x_0,r})|^2\,\mathrm{d}x \sim  \dashint_{B_r(x_0)}|{\bf V}_\varphi(\bfg)-({\bf V}_\varphi(\bfg))_{x_0,r}|^2\,\mathrm{d}x\,.
\label{eq:equivalencebis}
\end{equation}

The following version of Sobolev-Poincar\'e inequality for double phase problems has been proved in \cite[Theorem~2.13]{OK2017}.  
\begin{lemma}[Sobolev-Poincar\'e inequality]\label{thm:sob-poincare0}
Let $H$ be defined as in \eqref{eq:H} and \eqref{eq:pq}. 
Then there exists $\theta_0=\theta_0(n,p,q)\in(0,1)$ 
such that for any ${\bf w}\in W^{1,1}(\Omega;\R^N)$ and $B_r=B_r(x_0)\subset\Omega$ with $r\leq1$, we have 
$$
\dashint_{B_r} H\left(x, \frac{|{\bf w}-({\bf w})_{x_0,r}|}{r}\right)\,\mathrm{d}x\leq c(1+[a]_{C^{0,\alpha}}\|D{\bf w}\|^{q-p}_{L^p(B_r)}) \left(\dashint_{B_r} [H\left(x,{|D{\bf w}|}\right)]^{\theta_0}\,\mathrm{d}x\right)^\frac{1}{\theta_0}\,,
$$
for some $c=c(n,p,q)\geq1$. 
\end{lemma}

Let $x^-_{x_0,r}, x^+_{x_0,r}\in \overline{B_r(x_0)}$ be such that
\begin{equation*}
a^-_{x_0,r}:=a(x^-_{x_0,r}) = \inf_{x\in B_r(x_0)}a(x)
\quad\text{and}\quad
a^+_{x_0,r}:=a(x^+_{x_0,r}) = \sup_{x\in B_r(x_0)}a(x)\,. 
\end{equation*}
Then we write
\begin{equation}
\label{eq:Hpmdef}
H^-_{B_r(x_0)}(t):=H(x^-_{x_0,r},t)
\quad\text{and}\quad
H^+_{B_r(x_0)}(t):=H(x^+_{x_0,r},t) \,. 
\end{equation}

  In order to obtain a Sobolev-Poincar\'e inequality for the shifted function $H_{|{\bf Q}|}$ we need an \emph{a priori} higher integrability assumption on the gradient. 
%
\begin{lemma}[Sobolev-Poincar\'e inequality]\label{thm:sob-poincare}
Let $H$ be defined as in \eqref{eq:H} and \eqref{eq:pq}, and let ${\bf Q}\in\R^{N\times n}$, with ${\bf Q}\neq {\bf 0}$. 
Then there exist $\theta=\theta(n,p,q)\in(0,1)$ 
such that for any ${\bf w}\in W^{1,1}(\Omega;\R^N)$ with $D {\bf w} \in L^{p(1+s_0)}_{\loc}(\Omega)$ for some $s_0>0$, and $B_r\Subset\Omega$ with $r\leq1$ satisfying $\|D {\bf w}  \|_{L^{p(1+s_0)}(B_r)}\le 1$, we have 
$$
\dashint_{B_r} H_{|{\bf Q}|}\left(x, \frac{|{\bf w}-({\bf w})_{B_r}|}{r}\right)\,\mathrm{d}x\leq c 
\left(\dashint_{B_r} \left[(H^-_{B_r})_{|{\bf Q}|}(|D{\bf w}|)\right]^{\theta}\,\mathrm{d}x\right)^\frac{1}{\theta} +c (r^{\alpha(s_0)} + r^\alpha|{\bf Q}|^{q-p})|\bfQ|^p\,,
$$
for some $c=c(n,p,q,\alpha,[a]_{C^{0,\alpha}})\geq1$, where 
$\alpha(s_0):= \alpha-\frac{(q-p)n}{p(1+s_0)} >0$.
\end{lemma}
\begin{proof}
Since
$$\begin{aligned}
H_{|\bfQ|}(x,t) & \sim  (t+|\bfQ|)^{p-2}t^2 + a(x)(t+|\bfQ|)^{q-2}t^2 \\
& \lesssim (t+|\bfQ|)^{p-2}t^2 + a^-_{B_r}(t+|\bfQ|)^{q-2}t^2 + r^\alpha(t+|\bfQ|)^{q-2}t^2\\
& \lesssim  (H_{B_r}^-)_{|\bfQ|}(t) + r^\alpha(t^q+|\bfQ|^q),
\end{aligned}$$
we can estimate, by using the Sobolev-Poincarè inequality for shifted $N$-function $(H_{B_r}^-)_{|\bfQ|}(t)$ (see, e.g., \cite[Theorem 7]{DIEETT08}) and for function $\varphi(t)=t^q$, 
$$\begin{aligned}
\dashint_{B_r} H_{|{\bf Q}|}\left(x, \frac{|{\bf w}-({\bf w})_{B_r}|}{r}\right)\,\mathrm{d}x
&\lesssim \dashint_{B_r} (H^-_{B_r})_{|{\bf Q}|}\left(\frac{|{\bf w}-({\bf w})_{B_r}|}{r}\right)\,\mathrm{d}x + r^\alpha \dashint_{B_r} \left(\frac{|{\bf w}-({\bf w})_{B_r}|}{r}\right)^q\,\mathrm{d}x+r^\alpha |{\bf Q}|^q \\
&\lesssim  \left(\dashint_{B_r} \left[(H^-_{B_r})_{|{\bf Q}|}(|D{\bf w}|)\right]^{\theta_1}\,\mathrm{d}x\right)^\frac{1}{\theta_1}+ r^\alpha\left(\dashint_{B_r} |D{\bf w}|^{q_*}\, \mathrm{d}x\right)^{\frac{q}{q_*}} + r^\alpha|{\bf Q}|^q,
\end{aligned}$$
where $\theta_1\in(0,1)$ and $q_*:=\min\{1,\frac{nq}{n+q}\}<p$. Set  $\theta:=\max\{\theta_1,q_*/p\}\in(0,1)$. Note that by H\"older's inequality
$$\begin{aligned}
r^\alpha\left(\dashint_{B_r} |D{\bf w}|^{q_*}\, \mathrm{d}x\right)^{\frac{q}{q_*}}
& \le r^\alpha\left(\dashint_{B_r} |D{\bf w}|^{p\theta}\, \mathrm{d}x\right)^{\frac{1}{\theta}} \left(\dashint_{B_r} |D{\bf w}|^{p(1+s_0)}\, \mathrm{d}x\right)^{\frac{q-p}{p(1+s_0)}}\\
&\le r^{\alpha-\frac{(q-p)n}{p(1+s_0)}} \|D{\bf w}\|_{L^{p(1+s_0)}(B_r)}^{q-p}\left(\dashint_{B_r} |D{\bf w}|^{p\theta}\, \mathrm{d}x\right)^{\frac{1}{\theta}}\,.
\end{aligned}
$$
Using this and the facts that $\|D{\bf w}\|_{L^{p(1+s_0)}(B_r)}\le 1$ and $\phi(t)\lesssim \phi_{|\bfQ|}(t)+\phi(|\bfQ|)$ with $\phi(t)=t^p$, we obtain 
$$\begin{aligned}
\dashint_{B_r} H_{|{\bf Q}|}\left(x, \frac{|{\bf w}-({\bf w})_{B_r}|}{r}\right)\,\mathrm{d}x
&\lesssim  \left(1+ r^{\alpha-\frac{(q-p)n}{p(1+s_0)}}  \right)\left(\dashint_{B_r} \left[(H^-_{B_r})_{|{\bf Q}|}(|D{\bf w}|)\right]^{\theta}\,\mathrm{d}x\right)^\frac{1}{\theta}\\
&\qquad +r^{\alpha-\frac{(q-p)n}{p(1+s_0)}}   |\bfQ|^{p} + r^\alpha|{\bf Q}|^q.
\end{aligned}
$$
This completes the proof.
\end{proof}

{The following lemma is useful to derive a higher integrability result. It is a variant of the results by Gehring~\cite{Gehring} and Giaquinta-Modica~\cite[Theorem~6.6]{giustibook}.
\begin{lemma}\label{lem:gehring}
Let $B_0\subset\R^n$ be a ball, $f\in L^1(B_0)$, and $g\in L^{s_0}(B_0)$ for some $s_0>1$. Assume that for some $\gamma\in(0,1)$, $c_1>0$ and all balls $B$ with $2B\subset B_0$
\begin{equation*}
\dashint_B |f|\,\mathrm{d}x\leq c_1 \left(\dashint_{2B}|f|^\gamma\,\mathrm{d}x\right)^{1/\gamma} + \dashint_{2B}|g|\,\mathrm{d}x\,.
\end{equation*}
Then there exist $s_1>1$ and $c_2>1$ such that $g\in L^{s_1}_{\rm loc}(B)$ and for all $s_2\in[1,s_1]$
\begin{equation*}
\left(\dashint_{B}|f|^{s_2}\,\mathrm{d}x\right)^{1/{s_2}}\leq c_2 \dashint_{2B}|f|\,\mathrm{d}x + c_2 \left(\dashint_{2B}|g|^{s_2}\,\mathrm{d}x\right)^{1/{s_2}}\,.
\end{equation*}
\end{lemma}}

We conclude this section with the following useful lemma about an almost concave condition, see \cite[Lemma~2.2]{OKJFA18}.

\begin{lemma}
Let $\Psi:[0,\infty)\to[0,\infty)$ be non-decreasing and such that $t\to \frac{\Psi(t)}{t}$ be non-increasing. Then there exists a concave function $\widetilde{\Psi}: [0,\infty)\to[0,\infty)$ such that
\begin{equation*}
\frac{1}{2} \widetilde{\Psi}(t) \leq \Psi(t) \leq \widetilde{\Psi}(t) \quad \mbox{ for all $t\geq0$. }
\end{equation*}
\label{lem:lemma2.2ok}
\end{lemma}

\subsection{$\mathcal{A}$-harmonic and $\varphi$-harmonic functions}

Let $\mathcal{A}$ be a bilinear form on $\R^{N\times n}$. We say that $\mathcal{A}$ is {\em strongly elliptic in the sense of Legendre-Hadamard} if for all $\bm b \in \R^N$ and  $z \in\R^{n}$, it holds that 
\begin{equation}
  \nu_{\mathcal{A}} \abs{\bm b}^2 \abs{z}^2\leq \langle\mathcal{A}(\bm b \otimes z)\,|\,(\bm b \otimes z)\rangle\leq L_{\mathcal{A}} \abs{\bm b}^2 \abs{z}^2
\label{eq:LegHam}
\end{equation}
for some $L_{\mathcal{A}}\geq \nu_{\mathcal{A}}>0$. 
We say that a Sobolev function $\bfw$ on a ball~$B_R(x_0)$ is
\emph{$\mathcal{A}$-harmonic} on $B_R(x_0)$ if it satisfies $-\divergence (\mathcal{A}D \bfw)=0$ in the sense of distributions; i.e.,
\begin{equation*}
\int_{B_R(x_0)} \langle\mathcal{A}D{\bf w}\, |\, D{\bm \psi}\rangle\,\mathrm{d}x=0\,,\quad \mbox{ for all }{\bm \psi}\in C^\infty_0(B_R(x_0);\R^N)\,.
\end{equation*}
It is well known from the classical theory (see, e.g.~\cite[Chapter 10]{giustibook}) that ${\bf w}$ is smooth in the interior of $B_R(x_0)$, and it satisfies the estimate
\begin{equation}
\sup_{B_{R/2}(x_0)}|D {\bf w}|+R \sup_{B_{R/2}(x_0)}|D^2 {\bf w}| \leq c(n,N,\nu_{\mathcal{A}},L_{\mathcal{A}}) \dashint_{B_R(x_0)}|D{\bf w}| \,\mathrm{d}x\,.
\label{eq:westimate1}
\end{equation}
Moreover,  if  $\bfu\in W^{1,\phi}(B_R;\R^N)$, where $\phi$ is an $N$-function with $\phi, \phi^*\in \Delta_2$, then there exists a unique  $\mathcal{A}$-harmonic mapping $\bfw\in \bfu+W^{1,\phi}_0(B_R;\R^N)$ and we have the following Calder\'on-Zygmund type estimate (see for instance \cite[Theorem 18 and Remark 19]{DIELENSTROVER12}):
\begin{equation}
 \dashint_{B_R(x_0)}\phi(|D{\bf w}|) \leq c(n,N,\nu_{\mathcal A},L_{\mathcal A},\Delta_2(\psi,\psi^*)) \dashint_{B_R(x_0)}\phi(|D{\bf u}|) \,\mathrm{d}x\,.
\label{eq:westimate2}
\end{equation}

Let $\varphi\in C^{1}([0,\infty)) $ be an $N$-function satisfying \eqref{(2.1celok)}.  We say that a map ${\bf w}\in W^{1,\varphi}(B_R(x_0);\R^N)$ is \emph{$\varphi$-harmonic} on $B_\varrho(x_0)$ 
if ${\bf w}$ is a weak solution to the system
\begin{equation}
\, {\rm div}\left( \frac{\phi'(\kabs{D \bfw})}{|D \bfw|} \, D \bfw\right) \, = \, {\bf 0} \quad \text{in }\ B_{R}(x_0)\,,\
\label{system2}
\end{equation}
that is,
\begin{equation*}
\int_{B_R(x_0)} \left\langle\frac{\varphi'(|D{\bf w}|)}{|D{\bf w}|}D{\bf w}\, \bigg| \,D\bm \psi\right\rangle\,\mathrm{d}x=0\,,\quad \mbox{ for all }\bm \psi\in C^\infty_0(B_R(x_0);\R^N)\,.
\end{equation*}
We notice that if $\phi(t)=t^p+at^q$, then $\phi$ satisfies \eqref{(2.1celokbis)}, and $D{\bf w}$ and ${\bf V}_\phi (D{\bf w})$ are locally H\"older continuous due to the following results, see \cite[Proposition 2.4 and Theorem 2.5]{DIESTROVER09}. 
\begin{proposition}\label{lemma:2:G-harmonic-holder} Let 
$$
\phi(t)= t^p+at^q \quad \text{with }\ 1<p\le q \ \text{ and }\ a\ge0\,,
$$
and $\bfV_\varphi$ be defined as in \eqref{eq:defV}.
Then there exist a constant $c>0$ and an exponent $\gamma_0 \in (0,1)$ depending only on $n$, $N$, $p$ and $q$ (independent of $a$) such that if ${\bf w} \in W^{1,\phi}(B_{R}(x_0),\mathbb{R}^{N})$ is a weak solution to \eqref{system2},
 then  for every $\tau\in(0,1]$,  there hold
$$
\sup_{B_{\tau R/2}(x_0)} \phi(\kabs{D {\bf w}})  \leq  c \dashint_{B_{\tau R}(x_0)} \phi(\kabs{D {\bf w}})\,\mathrm{d}x\,, 
$$
and 
\begin{equation}\label{eq:excessdecayw}
\dashint_{B_{\tau R}(x_0)} \kabs{\bfV_\varphi(D{\bf w})-(\bfV_\varphi(D {\bf w}))_{x_0,{\tau R}}}^2 \, \mathrm{d}x  \leq  c \tau^{2 \gamma_0} 	\dashint_{B_R(x_0)} \kabs{\bfV_\varphi(D{\bf w})-(\bfV_\varphi(D {\bf w}))_{x_0,R}}^2 \, \mathrm{d}x\,.
\end{equation}
\end{proposition}


\subsection{Harmonic type approximation results} \label{sec:harmonicapproximation}

We recall here two different harmonic type approximation results. The first one is the \emph{$\mathcal{A}$-harmonic approximation}, which addresses the problem of finding an $\mathcal{A}$-harmonic function~${\bf w}$ which is close to a given a Sobolev function $\bfu$ on a ball~$B_r$. Such a function is the $\mathcal{A}$-harmonic function with the same boundary values as $\bfu$; i.e., a Sobolev function~${\bf w}$
which satisfies
\begin{equation}
  \label{eq:calA1}
  \begin{cases}
    -\divergence (\mathcal{A} D {\bf w})= {\bf 0} &\qquad\text{on $B_r$}
    \\
    {\bf w}= \bfu & \qquad\text{on $\partial B_r$}
  \end{cases}
\end{equation}
in the sense of distribution.


The following is the version of the $\mathcal{A}$-harmonic approximation stated in \cite[Lemma~2.7]{CeladaOk}, which relies on \cite[Theorem~14]{DIELENSTROVER12}. 
It is obtained, coupling the $\mathcal{A}$-harmonic approximation result proven in  \cite[Theorem~14]{DIELENSTROVER12} with the higher integrability result coming from Caccioppoli and Poincar\'e inequalities. 
\begin{lemma}
  \label{thm:Aappr_psi}
  Let  $\mathcal{A}$ be a strongly elliptic (in the sense of
  Legendre-Hadamard) bilinear form on $\R^{N\times n}$, $\phi$ be an N-function with $\phi, \phi^* \in \Delta_2$, and let $s>1$ and $\mu>0$.  
 Then for every  $\epsilon>0$, there exists $\delta>0$ depending on $n$, $N$,  $\nu_{\mathcal{A}}$, $L_{\mathcal{A}}$, $\Delta_2(\phi,\phi^*)$ and
  $s$ such that the following holds.  If $\bfu \in
  W^{1,\phi}(B_r;\R^N)$ satisfies
  $$
  \dashint_{B_r} \phi(|D\bfu|) \, \d x \le    \left(\dashint_{B_r} \phi(|D\bfu|)^s \, \d x \right)^{\frac{1}{s}} \le \phi(\mu)\,,
  $$
  and is an {\em almost $\mathcal{A}$-harmonic} in $B_r$ in the sense that
  \begin{align*}
    \biggabs{\dashint_{B_r} \langle\mathcal{A}D \bfu \, | \, D \bm \psi\rangle\,\mathrm{d}x}
    \leq \delta \mu
    \norm{D \bm \psi}_{\infty}
  \end{align*}
  for all $\bm \psi \in C^\infty_0(B_r;\R^N)$, then there holds
  \begin{equation*}
 \label{eq:Aappr_est}
    \dashint_{B_r} \phi\bigg(\frac{\abs{\bfu-{\bf w}}}{r}\bigg)\,\mathrm{d}x +
    \dashint_{B_r} \phi(\abs{D\bfu - D {\bf w}})\,\mathrm{d}x \leq \epsilon \phi(\mu)\,,
  \end{equation*}
  where ${\bf w} \in W^{1, \phi}(B_r;\RN)$ is the unique weak solution of~\eqref{eq:calA1}.
  \end{lemma}

%
Now, moving on to $\varphi$-harmonic mappings, the following \emph{$\varphi$-harmonic approximation} lemma (\cite[Lemma~1.1]{DIESTROVER10}) is the extension to general convex functions of the $p$-harmonic approximation lemma \cite{DUMIN04b}, \cite[Lemma~1]{DUMIN04}, and allows to approximate ``almost $\varphi$-harmonic'' mappings by $\varphi$-harmonic ones. In particular, we present the version introduced in  \cite[Corollary~2.10]{CeladaOk}.

\begin{lemma}\label{lem:phiharmapprox}
Let $\varphi$ be an $N$-function satisfying \eqref{(2.1celok)}, $s>1$, and $c_0>0$. For every $\varepsilon>0$ there exists $\delta>0$ depending only on $\varepsilon$, $n$, $N$, $p$, $q$, $s$ and $c_0$ such that the following holds. 
If ${\bf u}\in W^{1,\varphi}(B_r;\R^N)$ satisfies 
$$
\left(\dashint_{B_r} \phi(|D\bfu|)^s\,\d x\right)^{\frac{1}{s}} \le c_0 \dashint_{B_r} \phi(|D\bfu|)\,\d x\,,
$$
and is \emph{almost $\varphi$-harmonic} in the sense that
$$
\dashint_{B_r} \left\langle\frac{\varphi'(|D{\bf u}|)}{|D{\bf u}|}D{\bf u}\, \biggl|\, D\bm \psi\right\rangle\,\mathrm{d}x \leq \delta\left(\dashint_{B_{2r}}\varphi(|D{\bf u}|)\,\mathrm{d}x+\varphi(\|D\bm \psi\|_{\infty})\right)
$$
for all $\bm \psi\in C^\infty_0(B_r;\R^N)$, then the unique $\varphi$-harmonic ${\bf w}\in \bfu + W^{1,\varphi}_0(B_r;\R^N)$ satisfies
$$
\dashint_{B_r} |{\bf V}_{\phi}(D{\bf u})-{\bf V}_{\phi}(D{\bf w})|^{2}\,\mathrm{d}x\ \leq \varepsilon \dashint_{B_{2r}}\varphi(|D{\bf u}|)\,\mathrm{d}x\,,
$$
where ${\bf V}_\varphi$ is as in \eqref{eq:defV}. 
\end{lemma}

%

\section{Caccioppoli and reverse H\"older type estimates} \label{sec:caccio_reverse}

Throughout this section, let $H:\Omega\times[0,\infty)\to[0,\infty)$ be defined as in \eqref{eq:H} complying with \eqref{eq:pq}, and  ${\bf A}:\Omega\times \R^{N\times n}\to \R^{N\times n}$ comply with \eqref{eq:1.8ok1}--\eqref{eq:1.12ok}.

\subsection{Caccioppoli type estimates} \label{sec:caccioppolitype}

Let $B_r=B_r(x_0)\subset\Omega$, ${\bf Q}\in \R^{N\times n}$ and $\bm\ell_{x_0,r, {\bf Q}}$ be the affine function defined as 
$$
\bm\ell_{x_0,r,{\bf Q}}(x):= (\bfu)_{x_0,r} + {\bf Q}(x-x_0)\,, \quad x\in \R^n \,.
$$
The first key tool is the following Caccioppoli type estimate for $\bfu-\bm\ell_{x_0,r,{\bf Q}}$. 

\begin{lemma}\label{lem:lemma4.1} (Caccioppoli estimates)
Let $\bfu\in W^{1,1}(\Omega;\R^N)$ with $H(\cdot,|D\bfu|)\in L^1(\Omega)$ be a weak solution to  \eqref{system1}. 
Then $B_{2r}(x_0)\subset\Omega$ with $r\leq1$ and  ${\bf Q}\in\mathbb{R}^{N\times n}$ with $(2r)^\alpha|{\bf Q}|^{q-p}\le 1$, we have
\begin{equation}
\begin{split}
 \dashint_{B_{r}(x_0)} H_{|{\bf Q}|}(x,|D\bfu-{\bf Q}|)\,\mathrm{d}x  & \leq c \dashint_{B_{2r}(x_0)} H_{|{\bf Q}|}\left(x,\frac{|{\bf u}-\bm\ell_{x_0,r,{\bf Q}}|}{2r}\right)\,\mathrm{d}x \\
 &\qquad + c (r^{\beta_0}+r^\alpha |\bfQ|^{q-p})^{\frac{q}{q-1}}H^+_{B_{2r}(x_0)}(|\bfQ|)\,,
\end{split}
\label{eq:caccioppoliII}
\end{equation} 
for some constant $c=c(n,N,p,q,[a]_{C^{0,\alpha}},\nu,L)>0$, where $\alpha$ and $\beta_0$ are as in \eqref{eq:pq} and \eqref{eq:1.12ok}, respectively. 
\end{lemma}

\begin{proof}
The proof scheme is nowadays standard, compare, e.g., with the argument of \cite[Lemma~4.1]{OKJFA18}. 
We 
use the shorthands $\bm \ell_{r}$, $B_r$ and $x^-_r$ 
for $\bm \ell_{x_0,r,{\bf Q}}$, $B_r(x_0)$ and $x^-_{x_0,r}$, respectively. 
We consider a cut-off function $\eta\in C_0^\infty(B_{2r})$ such that $0\le \eta \le1$, $\eta\equiv1$ on $B_{r}$ and $|D\eta|\leq c(n)/r$, and, correspondingly, we define the function $\bm \psi:=\eta^q({\bf u} - \bm \ell_{r})$. 
Note that
\begin{equation}
D\bm \psi = \eta^q D({\bf u} - \bm \ell_{r}) + q\eta^{q-1} ({\bf u} - { \bm \ell_{r}})\otimes D \eta \,.
\label{eq:gradienttest}
\end{equation}
Taking $\bm \psi$ as a test function in \eqref{system1} 
and using the identity
\begin{equation}
\dashint_{B_{2r}}\langle {\bf A}(x^-_{2r},{\bf Q})\,|\, D\bm \psi\rangle\,\mathrm{d}x=0
\label{eq:4.4ok}
\end{equation}
we get
\begin{equation*}
\begin{split}
0= \dashint_{B_{2r}}\langle {\bf A}(x,D\bfu)-{\bf A}(x^-_{2r},{\bf Q}))\,|\,D\bm \psi \rangle\,\mathrm{d}x & = \dashint_{B_{2r}}\langle{\bf A}(x,D\bfu)-{\bf A}(x,{\bf Q}))\,|\,D\bm \psi \rangle\,\mathrm{d}x \\
& \,\,\,\,\,\, + \dashint_{B_{2r}}\langle {\bf A}(x,{\bf Q})-{\bf A}(x^-_{2r},{\bf Q}))\,|\,D\bm \psi\rangle\,\mathrm{d}x\,,
\end{split}
\end{equation*}
whence, taking into account \eqref{eq:gradienttest}, 
\begin{equation}
\begin{split}
J_1:&=\dashint_{B_{2r}} \eta^q \langle {\bf A}(x,D\bfu)-{\bf A}(x,{\bf Q})\,|\, D{\bf u} -{\bf Q}\rangle\,\mathrm{d}x \\
& =  \dashint_{B_{2r}}\langle{\bf A}(x_{2r}^-, {\bf Q})-{\bf A}(x,{\bf Q})\,|\, D{\bm \psi}\rangle\,\mathrm{d}x   - q \dashint_{B_{2r}}\eta^{q-1} \langle {\bf A}(x,D\bfu)-{\bf A}(x,{\bf Q}) \,|\, ({\bf u} - {\bm \ell_{r}})  \otimes D\eta \rangle\,\mathrm{d}x \\
& =: J_2+J_3\,.
\end{split}
\label{eq:4.2ok}
\end{equation}

Now, we proceed to estimate each term above separately. With \eqref{eq:1.9ok} and \eqref{(2.6b)} we get

\begin{equation}
J_1  \geq \frac{1}{\tilde{c}} \dashint_{B_{2r}} \eta^q H_{|{\bf Q}|}(x,|D{\bf u} - {\bf Q}|)\,\mathrm{d}x
\label{eq:stimaJ1}
\end{equation}
for some $\tilde{c}\ge 1$.
To estimate $J_3$ we use \eqref{eq:1.8ok1}, \eqref{eq:phi_shifted}, Lemma~\ref{technisch-mu}, 
Young's inequality  with $\phi(t)=H_{|{\bf Q}|}(x,t)$, \eqref{eq:2.6ok} and \eqref{(2.6a)} and we get

\begin{equation}
\begin{split}
|J_3| & \le c \dashint_{B_{2r}}\left(\int_0^1|D_\xi {\bf A}(x,\tau(D\bfu-{\bf Q}) + {\bf Q})|\,\mathrm{d}\tau\right)|D\bfu-{\bf Q}|\frac{|{\bf u} - \bm \ell_{r}|}{2r}\,\mathrm{d}x \\
& \le c  \dashint_{B_{2r}}\frac{H'(x, |{\bf Q}|+|D\bfu-{\bf Q}|)}{(|{\bf Q}|+|D\bfu-{\bf Q}|)}|D\bfu-{\bf Q}|\frac{|{\bf u} - \bm \ell_{r}|}{2r}\,\mathrm{d}x \\
& \le c \dashint_{B_{2r}}\frac{H_{|{\bf Q}|}(x,|D\bfu-{\bf Q}|)}{|D\bfu-\bfQ|} \frac{|{\bf u} - \bm \ell_{r}|}{2r}\,\mathrm{d}x \\
& \le \frac{1}{4\tilde c}\dashint_{B_{2r}}H_{|{\bf Q}|}(x,|D\bfu-{\bf Q}|)\,\mathrm{d}x+ c \dashint_{B_{2r}} H_{|{\bf Q}|}\left(x,\frac{|{\bf u} - \bm \ell_{r}|}{2r}\right)\,\mathrm{d}x\,.
\end{split}
\label{eq:stimaJ3}
\end{equation}
As for $J_2$, from \eqref{eq:1.11ok} 
and using Young's inequalities \eqref{eq:2.5ok} for $\phi(t)=H_{|\bfQ|}(x,t)$ for each $x\in\Omega$ and  $\phi(t)=\tilde\phi_{|\bfQ|}(t)$ with $\tilde\phi(t)=t^p$, 
\eqref{eq:6.23dieett},  
we obtain
\begin{equation}
\begin{split}
|J_{2}| & \lesssim  \dashint_{B_{2r}}\left(r^{\beta_0}H'(x,|\bfQ|)+r^\alpha |{\bf Q}|^{q-1}\right)|D{\bm \psi}|\,\mathrm{d}x \\
& \lesssim  \dashint_{B_{2r}}\left\{r^{\beta_0}+r^\alpha |{\bf Q}|^{q-p}\right\} (H_{|\bfQ|})'(x,|\bfQ|) |D{\bm \psi}|\,\mathrm{d}x \\
& \leq \frac{1}{4\tilde{c}}\dashint_{B_{2r}}\eta^q H_{|\bfQ|}(x,|D\bfu-{\bf Q}|)\,\mathrm{d}x + c \dashint_{B_{2r}}H_{|\bfQ|}\left(x,\frac{|\bfu-{\bm \ell}_{r}|}{2r}\right)\,\mathrm{d}x \\
& \qquad +  c \dashint_{B_{2r}}(H_{|\bfQ|})^*\left(x,\left\{r^{\beta_0}+r^\alpha |\bfQ|^{q-p}\right\} (H_{|\bfQ|})'(x,|\bfQ|) \right)\,\mathrm{d}x \\
& \leq \frac{1}{4\tilde{c}}\dashint_{B_{2r}}\eta^q H_{|\bfQ|}(x,|D\bfu-{\bf Q}|)\,\mathrm{d}x + c \dashint_{B_{2r}}H_{|\bfQ|}\left(x,\frac{|\bfu-{\bm \ell}_{r}|}{2r}\right)\,\mathrm{d}x \\
& \qquad + c(r^{\beta_0}+r^\alpha |\bfQ|^{q-p})^{\frac{q}{q-1}} H(x^+_{2r},|\bfQ|)  \,.
\end{split}
\label{eq:stimaJ2}
\end{equation}
Plugging the estimates \eqref{eq:stimaJ1}, \eqref{eq:stimaJ3} and \eqref{eq:stimaJ2} 
into \eqref{eq:4.2ok} and reabsorbing some terms we obtain \eqref{eq:caccioppoliII}. The proof of \eqref{eq:caccioppoliII} is then concluded. 
\end{proof}

As a consequence of Sobolev-Poincaré inequality Lemma~\ref{thm:sob-poincare0}, Lemma~\ref{lem:lemma4.1} for ${\bf Q}={\bf 0}$ and Gehring's lemma with increasing supports (Lemma~\ref{lem:gehring}), we deduce a  higher integrability result for $H(x, |D\bfu|)$:

\begin{lemma}\label{lem:high0} (Higher integrability)
Let $\bfu\in W^{1,1}(\Omega;\R^N)$ with $H(\cdot,|D\bfu|)\in L^1(\Omega)$ be a weak solution to  \eqref{system1}. There exist constants $\sigma_0>0$ and $c>0$ depending on $n$, $N$, $p$, $q$, $\nu$, $L$, $\alpha$ and $[a]_{C^{0,\alpha}}$ such that for any $B_{2r}\subset \Omega$ with $\|H(\cdot,|D\bfu|)\|_{L^1(B_{2r})}\le 1$, 
we have
\begin{equation*}
\begin{split}
&\bigg(\dashint_{B_{r}(x_0)} \left[H(x, |D\bfu|)\right]^{1+\sigma_0}\,\mathrm{d}x \bigg)^{\frac{1}{1+\sigma_0}}\leq c \dashint_{B_{2r}(x_0)} H(x, |D\bfu|)\,\mathrm{d}x \,.
\end{split}
\end{equation*} 
Moreover, for every $t\in (0,1]$ there exists  $c_{t}=c_{t}(n,N,p,q,\nu,L,\alpha,[a]_{C^{0,\alpha}},t)>0$ such that
\begin{equation}
\begin{split}
&\bigg(\dashint_{B_{r}(x_0)} \left[H(x, |D\bfu|)\right]^{1+\sigma_0}\,\mathrm{d}x \bigg)^{\frac{1}{1+\sigma_0}}\leq c_{t}\bigg(\dashint_{B_{2r}(x_0)} H(x, |D\bfu|)^t\,\mathrm{d}x\bigg)^{\frac{1}{t}}\,.
\end{split}
\label{eq:caccioppoliIbist0}
\end{equation} 
\end{lemma}

\begin{remark} \label{rmk:smallness}
Lemma \ref{lem:high0} implies $H(\cdot,|D{\bfu}|) \in L^{1+\sigma_0}_{\mathrm{loc}}(\Omega)$. Then for each $\Omega'\Subset \Omega$, there exists $r_0 \in (0,1]$ such that  for any $B_{2r}
(x_0)\subset \Omega'$ with $r\in(0,r_0]$,
\begin{equation}\label{eq:assless1}
|B_{2r}(x_0)|\le 1
\quad \text{and}\quad
\int_{B_{2r}(x_0)} H(x,|D{\bfu}|)^{1+\sigma_0} \,\mathrm{d}x \le 1 \,.
\end{equation}
\end{remark}

\begin{lemma}\label{cor:reverse0}
Let $\bfu\in W^{1,1}(\Omega;\R^N)$ with $H(\cdot,|D\bfu|)\in L^1(\Omega)$ be a weak solution to  \eqref{system1}, and let $\sigma_0>0$ be the exponent of Lemma~\ref{lem:high0}. There exists a  constant $c=c(n,N,p,q,\nu,L,\alpha,[a]_{C^{0,\alpha}})>0$ such that for any $B_{2r}(x_0)\Subset \Omega$ satisfying \eqref{eq:assless1} with $r\le 1/2$, we have
\begin{equation}\label{eq:Hreverse0}
\begin{split}
&\bigg(\dashint_{B_{r}(x_0)} [H(x, |D\bfu|)]^{1+\sigma_0}\,\mathrm{d}x\bigg)^{\frac{1}{1+\sigma_0}}\leq c H^-_{B_{2r}(x_0)}\bigg( \dashint_{B_{2r}(x_0)}|D\bfu|\,\mathrm{d}x\bigg) \,.
\end{split}
\end{equation}
In particular, we have  
\begin{equation}\label{eq:H'reverse}
\dashint_{B_r(x_0)} (H^-_{B_{2r}(x_0)})'(|D\bfu |)\,\mathrm{d} x\le  c (H^{-}_{B_{2r}(x_0)})'\bigg(\dashint_{B_{2r}(x_0)} |D\bfu |\,\mathrm{d} x \bigg)\,.
\end{equation}
\end{lemma}
\begin{proof} 
For simplicity, we omit writing the center $x_0$ in the proof, and use the shorthands $H^\pm_{2r}$ in place of  $H^\pm_{B_{2r}(x_0)}$. We first note that \eqref{eq:assless1} and Young's inequality imply
\begin{equation}\label{eq:assless11}
\int_{B_{2r}} H(x,|D\bfu|) \,\mathrm{d}x \le \frac{1}{1+\sigma_0}\int_{B_{2r}} H(x,|D\bfu|)^{1+\sigma_0} \,\mathrm{d}x + \frac{\sigma_0}{1+\sigma_0} |B_{2r}| \le 1\,.
\end{equation}
Therefore, we  obtain \eqref{eq:caccioppoliIbist0} which yields, for $t=\frac{1}{q}$, 
\begin{equation}
\begin{split}
&\bigg(\dashint_{B_{r}} \left[H(x, |D\bfu|)\right]^{1+ \sigma_0}\,\mathrm{d}x \bigg)^{\frac{1}{1+\sigma_0}}\leq  c \bigg(\dashint_{B_{2r}} [H^+_{2r}(|D\bfu|)]^\frac{1}{q}\,\mathrm{d}x\bigg)^q\,.
\end{split}
\label{eq:firstestimate0}
\end{equation}
Now, since the function $\Psi(t):=[H^+_{2r}(t)]^\frac{1}{q}$ complies with the assumptions of Lemma \ref{lem:lemma2.2ok}, by Jensen's inequality we conclude that
\begin{equation}
\begin{split}
&\bigg(\dashint_{B_{2r}} [H^+_{2r}(|D\bfu|)]^\frac{1}{q}\,\mathrm{d}x\bigg)^q\\
& \lesssim  H^+_{2r}\bigg(\dashint_{B_{2r}}|D\bfu|\,\mathrm{d}x\bigg)
\lesssim  H^-_{2r}\bigg(\dashint_{B_{2r}}|D\bfu|\,\mathrm{d}x\bigg) + r^\alpha \bigg(\dashint_{B_{2r}}|D\bfu|\,\mathrm{d}x \bigg)^{q} \,.  
\end{split}
\label{eq:H1/q0}
\end{equation}
On the other hand,  using  H\"older's inequality, \eqref{eq:assless1} and \eqref{eq:pq}  we have 
\begin{equation}\label{eq:estimq-p0}
r^\alpha (|D\bfu|)_{2r}^{q-p}   \le r^\alpha (|D\bfu|^p)_{2r}^{\frac{q-p}{p}} \le r^\alpha \left(|D\bfu|^{p(1+\sigma_0)}\right)_{2r}^{\frac{q-p}{p(1+\sigma_0)}}
 \lesssim r^{\alpha-n\frac{q-p}{p(1+\sigma_0)}} \le 1 \,.
\end{equation}
Consequently, applying the previous inequalities \eqref{eq:H1/q0} and  \eqref{eq:estimq-p0}  
 to \eqref{eq:firstestimate0},
we obtain \eqref{eq:Hreverse0}.

Furthermore, since 
$$
\big((H^-_{B_{2r}})'\circ (H^-_{2r})^{-1} \big) (t) \sim  \frac{t}{(H^-_{B_{2r}})^{-1} (t)}
\quad \text{and} \quad
t \, \mapsto \,   \frac{1}{(H^-_{2r})^{-1} (t)} \ \text{ is non-increasing},
$$
by Lemma \ref{lem:lemma2.2ok}  with $\Psi(t)= t/(H^-_{2r})^{-1} (t)$, Jensen's inequality and  \eqref{eq:Hreverse0}, 
the estimate \eqref{eq:H'reverse} follows as:
$$\begin{aligned}
\dashint_{B_r} (H^-_{2r})'(|D\bfu |)\,\mathrm{d} x & \le c\left((H^-_{B_{2r}})'\circ (H^-_{2r})^{-1}\right) \bigg(\dashint_{B_r} H^-_{2r}(|D\bfu |)\,\mathrm{d} x\bigg)\\
&\le c\left((H^-_{2r})'\circ (H^-_{2r})^{-1}\right) \bigg(\dashint_{B_r} H(x, |D\bfu |)\,\mathrm{d} x\bigg) \\
&\le c (H^-_{2r})' \bigg(\dashint_{B_{2r}} |D\bfu |\,\mathrm{d} x\bigg)\,. 
\end{aligned}$$
This completes the proof.
\end{proof}

{Now, we are in position to establish a higher integrability result for $H_{|{\bf Q}|}(x, |D\bfu-{\bf Q}|)$, ${\bf Q}\neq {\bf 0}$. 
By Lemma \ref{lem:high0} we know that $D{\bf u}\in L^{p(1+\sigma_0)}_{\mathrm{loc}}(\Omega;\R^{N\times n})$, and, in view of Remark \ref{rmk:smallness}, we can find balls $B_r$ such that $\|D {\bf u}  \|_{L^{p(1+\sigma_0)}(B_r)}\le 1$. Thus, for such balls the Sobolev-Poincaré inequality of Lemma~\ref{thm:sob-poincare} holds with $s_0=\sigma_0$. Recall the constant $\alpha(s_0)$ in the lemma.
The following higher integrability result then follows again 
from Lemma~\ref{lem:lemma4.1} and Lemma~\ref{lem:gehring}: 
\begin{lemma}\label{lem:high}  (Improved higher integrability)
Let $\bfu\in W^{1,1}(\Omega;\R^N)$ with $H(\cdot,|D\bfu|)\in L^1(\Omega)$ be a weak solution to  \eqref{system1}. There exist constants $\sigma>0$ and $c>0$ depending on $n$, $N$, $p$, $q$, $\nu$, $L$, $\alpha$ and $[a]_{C^{0,\alpha}}$ such that for any $B_{2r}\subset \Omega$ with $\|H(\cdot,|D\bfu|)\|_{L^1(B_{2r})}\le 1$ and ${\bf Q}\in\mathbb{R}^{N\times n}$ with $0<(2r)^\alpha|{\bf Q}|^{q-p}\le 1$, we have
\begin{equation}
\begin{split}
&\bigg(\dashint_{B_{r}(x_0)} \left[H_{|{\bf Q}|}(x, |D\bfu-{\bf Q}|)\right]^{1+\sigma}\,\mathrm{d}x \bigg)^{\frac{1}{1+\sigma}} \\
&\leq c \dashint_{B_{2r}(x_0)} H_{|{\bf Q}|}(x, |D\bfu-{\bf Q}|)\,\mathrm{d}x   + c 
\left[(r^{\beta_0}+r^\alpha |\bfQ|^{q-p})^{\frac{q}{q-1}}+r^{\alpha_0} + r^\alpha|{\bf Q}|^{q-p}\right]
H^+_{B_{2r}(x_0)}(|\bfQ|) \,,
\end{split}
\label{eq:caccioppoliIbis}
\end{equation}
where 
\begin{equation}\label{alpha0}
\alpha_0:=\alpha(\sigma_0)= \alpha-\frac{(q-p)n}{p(1+\sigma_0)} .
\end{equation} 

Moreover, for every $t\in (0,1]$ there exists  $c_{t}=c_{t}(n,N,p,q,\nu,L,\alpha,[a]_{C^{0,\alpha}},t)>0$ such that
\begin{equation}
\begin{split}
\bigg(\dashint_{B_{r}(x_0)} \left[H_{|{\bf Q}|}(x, |D\bfu-{\bf Q}|)\right]^{1+\sigma}\,\mathrm{d}x \bigg)^{\frac{1}{1+\sigma}} 
&\leq c_{t} \Bigg\{\bigg(\dashint_{B_{2r}(x_0)} H_{|{\bf Q}|}(x, |D\bfu-{\bf Q}|)^t\,\mathrm{d}x\bigg)^{\frac{1}{t}} \\
&\qquad   +  \left[(r^{\beta_0}+r^\alpha |\bfQ|^{q-p})^{\frac{q}{q-1}}+r^{\alpha_0} + r^\alpha|{\bf Q}|^{q-p}\right] H^+_{B_{2r}(x_0)}(|\bfQ|)\Bigg\}\,.
\end{split}
\label{eq:caccioppoliIbist}
\end{equation} 
\end{lemma}}

\subsection{Reverse H\"older type estimates}

By the higher integrability result in Lemma~\ref{lem:high}, under assumption \eqref{eq:assless1}, we can obtain from  the following reverse H\"older type estimates for $|D\bfu - {\bf Q}|$ with the shifted $N$-function $H_{|{\bf Q}|}$ when ${\bf Q}=(D\bfu)_{x_0,2r}$.

\begin{lemma}\label{cor:reverse}
Let $\bfu\in W^{1,1}(\Omega;\R^N)$ with $H(\cdot,|D\bfu|)\in L^1(\Omega)$ be a weak solution to  \eqref{system1}, and let $\sigma>0$ be the exponent of Lemma~\ref{lem:high}. There exists a  constant $c=c(n,N,p,q,\nu,L,\alpha,[a]_{C^{0,\alpha}})>0$ such that for any $B_{2r}\Subset \Omega$ satisfying \eqref{eq:assless1} with $r\le 1/2$
and for ${\bf Q}= (D\bfu)_{x_0,2r}$, we have
\begin{equation}\label{eq:Hreverse}
\begin{split}
&\bigg(\dashint_{B_{r}(x_0)} [H_{|{\bf Q}|}(x, |D\bfu-{\bf Q}|)]^{1+\sigma}\,\mathrm{d}x\bigg)^{\frac{1}{1+\sigma}} \\
& \qquad \leq c (H^-_{B_{2r}(x_0)})_{|{\bf Q}|}\bigg( \dashint_{B_{2r}(x_0)}|D\bfu-{\bf Q}|\,\mathrm{d}x\bigg)
  + c  r^{\alpha_1} H^-_{B_{2r}(x_0)}(|\bfQ|)\,,
\end{split}
\end{equation}
where
\begin{equation}\label{alpha1}
\alpha_1:=\min\left\{ \frac{\beta_0q}{q-1},\alpha_0\right\}\,,
\end{equation}
and the constants $q$, $\beta_0$ and $\alpha_0$ are from \eqref{eq:H}, \eqref{eq:1.12ok} and \eqref{alpha0}, respectively.
\end{lemma}
\begin{proof} 
We adopt here the same notation and perform a similar argument as for Lemma \ref{cor:reverse0}. Again, as in \eqref{eq:assless11}, Young's inequality implies $\int_{B_{2r}(x_0)} H(x,|D\bfu|) \,\mathrm{d}x \leq 1$. 
From this we also deduce that $\|D\bfu\|_{L^{p}(B_{2r})}\leq1$, whence using H\"older's inequality, \eqref{eq:pq} and the facts that $2r\le 1$ and $|B_1|>1$, we obtain
$$
(2r)^\alpha |\bfQ|^{q-p} \leq r^\alpha (|D\bfu|^p)_{2r}^{\frac{q-p}{p}} \leq  (2r)^{\alpha-n\frac{q-p}{p}} |B_1|^{-\frac{q-p}{p}} \leq 1\,. 
$$
Therefore, when ${\bf Q}= (D\bfu)_{2r}$, we  obtain \eqref{eq:caccioppoliIbist} which yields, for $t=\frac{1}{q}$, 
\begin{equation}
\begin{split}
&\bigg(\dashint_{B_{r}} \left[H_{|{\bf Q}|}(x, |D\bfu-{\bf Q}|)\right]^{1+\sigma}\,\mathrm{d}x \bigg)^{\frac{1}{1+\sigma}} \\
&\leq c \bigg(\dashint_{B_{2r}} [(H^+_{2r})_{|{\bf Q}|}(|D\bfu-{\bf Q}|)]^\frac{1}{q}\,\mathrm{d}x\bigg)^q  +c \left[(r^{\beta_0}+r^\alpha |\bfQ|^{q-p})^{\frac{q}{q-1}}+r^{\alpha_0} + r^\alpha|{\bf Q}|^{q-p}\right]  H^+_{2r}(|\bfQ|)\,.
\end{split}
\label{eq:firstestimate}
\end{equation}
Now, since the function $\Psi(t):=[(H^+_{2r})_{|{\bf Q}|}(t)]^\frac{1}{q}$ complies with the assumptions of Lemma \ref{lem:lemma2.2ok}, by Jensen's inequality we conclude that
\begin{equation}
\begin{split}
&\bigg(\dashint_{B_{2r}} [(H^+_{2r})_{|{\bf Q}|}(|D\bfu-{\bf Q}|)]^\frac{1}{q}\,\mathrm{d}x\bigg)^q\\
& \lesssim  (H^+_{2r})_{|{\bf Q}|}\bigg(\dashint_{B_{2r}}|D\bfu-{\bf Q}|\,\mathrm{d}x\bigg)
\lesssim  (H^-_{2r})_{|{\bf Q}|}\bigg(\dashint_{B_{2r}}|D\bfu-{\bf Q}|\,\mathrm{d}x\bigg) + r^\alpha \bigg(\dashint_{B_{2r}}|D\bfu-{\bf Q}|\,\mathrm{d}x + |\bfQ|\bigg)^{q} \,.  
\end{split}
\label{eq:H1/q}
\end{equation}
On the other hand,  using  H\"older's inequality, \eqref{eq:assless1} and \eqref{eq:pq}  we have 
\begin{equation}\label{eq:estimq-p}
r^\alpha |\bfQ|^{q-p} \le r^\alpha (|D\bfu|)_{2r}^{q-p}   \le r^\alpha (|D\bfu|^p)_{2r}^{\frac{q-p}{p}} \le r^\alpha \left(|D\bfu|^{p(1+\sigma_0)}\right)_{2r}^{\frac{q-p}{p(1+\sigma_0)}}
 \lesssim r^{\alpha-n\frac{q-p}{p(1+\sigma_0)}} = r^{\alpha_0}\,.
\end{equation}
Thus, using \eqref{(2.6c)}
it holds that
\begin{equation}
\begin{split}
&r^\alpha \bigg(\dashint_{B_{2r}}|D\bfu-{\bf Q}|\,\mathrm{d}x + |\bfQ|\bigg)^{q}  
\lesssim r^{\alpha_0} \bigg(\dashint_{B_{2r}}|D\bfu-{\bf Q}|\,\mathrm{d}x + |\bfQ|\bigg)^{p}  \\
&\lesssim r^{\alpha_0} \bigg\{(H^-_{2r})_{|{\bf Q}|}\bigg( \dashint_{B_{2r}(x_0)}|D\bfu-{\bf Q}|\,\mathrm{d}x\bigg) + (H^-_{B_{2r}(x_0)}) (|{\bf Q}|) \bigg\}
\end{split}
\label{eq:Hshiftestim}
\end{equation}
Consequently, applying the previous inequalities \eqref{eq:H1/q}, \eqref{eq:estimq-p}  and  \eqref{eq:Hshiftestim}  to \eqref{eq:firstestimate}, 
we obtain \eqref{eq:Hreverse}. This concludes the proof.
\end{proof}

Finally, we derive comparison estimates concerned with the functions $H(x,t)$ and $H_{B_{2r}}^-(t)$. Note that $H(x,t)-H_{B_{2r}}^-(t)=(a(x)-a^-_{{2r}})t^q$ and  $H'(x,t)-(H_{B_{2r}}^-)'(t)=q(a(x)-a^-_{{2r}})t^{q-1}$.

\begin{lemma}\label{lem:comparison}
Let $\bfu\in W^{1,1}(\Omega;\R^N)$ with $H(\cdot,|D\bfu|)\in L^1(\Omega)$ be a weak solution to \eqref{system1}. There exists $\alpha_2=\alpha_2(n,N,p,q,\nu,L,\alpha,[a]_{C^{0,\alpha}})>0$ such that for any $B_{2r}(x_0)\Subset \Omega$ satisfying \eqref{eq:assless1} with $r\le 1/2$, we have
\begin{equation}\label{eq:comparison1}
 \dashint_{B_r(x_0)} (a(x)-a^-_{x_0,2r}) |D\bfu|^{q-1}\,\mathrm{d} x  \le cr^{\alpha_2} (H^{-}_{B_{2r}(x_0)})' \bigg( \dashint_{B_{2r}(x_0)} |D\bfu | \,\mathrm{d} x \bigg)
\end{equation}
and
\begin{equation}\label{eq:comparison2}
 \dashint_{B_r(x_0)} (a(x)-a^-_{x_0,2r}) |D\bfu|^{q}\,\mathrm{d} x  \le cr^{\alpha_2} H^{-}_{B_{2r}(x_0)}\bigg( \dashint_{B_{2r}(x_0)} |D\bfu | \,\mathrm{d} x \bigg)
\end{equation}
for some $c=c(n,N,p,q,\nu,L,\alpha,[a]_{C^{0,\alpha}})>0$. 
\end{lemma}
\begin{proof}
To enlighten the notation, set $B_\rho=B_\rho(x_0)$, $a^\pm_{\rho} = a^\pm_{x_0, \rho}$, $H^\pm_\rho:=H^\pm_{B_\rho(x_0)}$ and $H^{+,*}_\rho:=(H^+_{B_\rho(x_0)})^*$. Let $\bar{\sigma} := \frac{n\sigma_0}{2(1+\sigma_0)}>0$, where $\sigma_0>0$ is from Lemma~\ref{lem:high0}, and set positive constants
$$
\sigma_1:= \alpha+(-n+\bar{\sigma})\frac{q-p}{p}
\quad \mbox{ and } \quad 
\sigma_2:= \left\{n\sigma_0-\bar{\sigma}(1+\sigma_0)\right\}\frac{\sigma_0(p-1)}{p(1+\sigma_0)}\,.
$$

Now, we have to distinguish between two cases, depending on whether condition $H(x,|D\bfu|)\leq r^{-n+\bar{\sigma}}$ is satisfied or not.  Set
\begin{equation*}
E:=\{H(x,|D\bfu|)\leq r^{-n+\bar{\sigma}}\} \cap B_{r}
 \quad\text{and}\quad
  F:=B_{r} \backslash E\,,
\end{equation*}
and split the integral on the left hand side of \eqref{eq:comparison1} as
\begin{equation} \label{eq:JEF}
\begin{split}
\dashint_{B_r} (a(x)-a^-_{{2r}}) |D\bfu|^{q-1}\mathrm{d} x & = \dashint_{B_r}\mathbbm{1}_{E}(x)(a(x)-a^-_{{2r}})|D\bfu|^{q-1}\,\mathrm{d}x \\ 
&\qquad + \dashint_{B_{r}}\mathbbm{1}_{F}(x)(a(x)-a^-_{{2r}})|D\bfu|^{q-1}\,\mathrm{d}x \\
& =: J_{E} + J_{F} \,. 
\end{split}
\end{equation}
Note that on $E$ it holds that $|D\bfu|^p\leq r^{-n+\bar{\sigma}}$. We then have, with \eqref{eq:H'reverse},
\begin{equation}
\begin{split}
|J_{E}|  \lesssim \dashint_{B_{r}}\mathbbm{1}_{E}(x)r^{\alpha+(-n+\bar{\sigma})\frac{q-p}{p}} |D\bfu |^{p-1}\,\mathrm{d}x 
& \leq \frac{r^{\sigma_1}}{p} \dashint_{B_r} (H^-_{{2r}})'(|D\bfu |)\,\mathrm{d} x \\
&  \leq c r^{\sigma_1} (H^{-}_{{2r}})'\bigg(\dashint_{B_{2r}} |D\bfu |\,\mathrm{d} x \bigg)\,.
\end{split}
\label{eq:estimEtilde}
\end{equation} 
Now, we turn to the estimate of $J_{F}$. Using Jensen's inequality for $H^{+,*}_{{2r}}$, the fact that $H^{+,*}_{{2r}}(t) \leq H^*(x,t)$ for every $x\in B_{2r}$ and $t>0$, and recalling that $\varphi(t):=H(x,t)$, with fixed $x$, complies with \eqref{eq:2.6ok} and \eqref{ineq:phiast_phi_p}, we get 
\begin{equation*}
\begin{split}
|J_{F}| & \lesssim \dashint_{B_{r}}\mathbbm{1}_{F}(x)H'(x,|D\bfu |) \,\mathrm{d}x \\
& \lesssim (H^{+,*}_{{2r}})^{-1}\bigg(\dashint_{B_{r}}\mathbbm{1}_{F}(x)H^{+,*}_{{2r}}(H'(x,|D\bfu |)) \,\mathrm{d}x  \bigg) \\
& \lesssim (H^{+,*}_{{2r}})^{-1}\bigg(\dashint_{B_{r}}\mathbbm{1}_{F}(x)H(x,|D\bfu |) \,\mathrm{d}x \bigg)\,. 
\end{split}
\end{equation*} 
Now, on the set $F$ we have $r^{n-\bar{\sigma}}H(x,|D \bfu|)>1$. This, combined with \eqref{eq:Hreverse0} 
and the second inequality in \eqref{eq:assless1} gives
\begin{equation*}
\begin{split}
|J_{F}| & \lesssim (H^{+,*}_{{2r}})^{-1}\bigg(\dashint_{B_{r}}\mathbbm{1}_{F}(x)H(x,|D\bfu |) \,\mathrm{d}x \bigg) \\
 &  \lesssim (H^{+,*}_{{2r}})^{-1}\bigg(r^{\sigma_0(n-\bar{\sigma})}\dashint_{B_{r}}[H(x,|D\bfu |)]^{1+\sigma_0} \,\mathrm{d}x \bigg) \\
& = c (H^{+,*}_{{2r}})^{-1}\Bigg(r^{\sigma_0(n-\bar{\sigma})}\bigg(\dashint_{B_{r}}[H(x,|D\bfu |)]^{1+\sigma_0} \,\mathrm{d}x \bigg)^{\frac{1}{1+\sigma_0}+ \frac{\sigma_0}{1+\sigma_0}}\Bigg) \\
& \lesssim(H^{+,*}_{{2r}})^{-1}\Bigg(r^{\sigma_0(n-\bar{\sigma})-\frac{n\sigma_0}{1+\sigma_0}}\bigg(\dashint_{B_{r}}[H(x,|D\bfu |)]^{1+\sigma_0} \,\mathrm{d}x \bigg)^{\frac{1}{1+\sigma_0}}\Bigg) \\
& \leq r^{\sigma_2} (H^{+,*}_{{2r}})^{-1}\Bigg(H^-_{{2r}}\bigg(\dashint_{B_{2r}}|D\bfu |\,\mathrm{d}x\bigg)\Bigg) \\
& \leq r^{\sigma_2} \left((H^{+,*}_{{2r}})^{-1}\circ H^+_{{2r}}\right)\bigg(\dashint_{B_{2r}}|D\bfu |\,\mathrm{d}x\bigg)\,.
\end{split}
\end{equation*}
Moreover, since  $(H^{+,*}_{{2r}})^{-1}(H^+_{{2r}}(t))  \sim (H^+_{{2r}})'(t)$ by \eqref{eq:2.6ok}, using 
\eqref{eq:estimq-p} we have 
\begin{equation*}
(H^{+,*}_{{2r}})^{-1}\Bigg(H^+_{{2r}}\bigg( \dashint_{B_r} |D\bfu | \,\mathrm{d} x \bigg) \Bigg) \lesssim(H^{-}_{{2r}})' \bigg( \dashint_{B_r} |D\bfu | \,\mathrm{d} x \bigg)\,.    
\end{equation*}
We then have
\begin{equation}
\begin{split}
|J_{F}| & \lesssim r^{\sigma_2} (H^{-}_{{2r}})' \bigg( \dashint_{B_{2r}} |D\bfu | \,\mathrm{d} x \bigg) \,. 
\end{split}
\label{eq:estimFtilde}
\end{equation}
Therefore, combining the above estimates \eqref{eq:JEF}, \eqref{eq:estimEtilde} and \eqref{eq:estimFtilde}, and letting $\alpha_2\le \min\{\sigma_1,\sigma_2\}$, we obtain  \eqref{eq:comparison1}. The estimate \eqref{eq:comparison2} follows a similar, and even slightly simpler, argument. Hence, we omit its proof. 
\end{proof}


\section{Decay estimates for excess functionals} \label{sec:decayestimate}


Let $\bfu\in W^{1,1}(\Omega;\R^N)$ with $H(\cdot,|D\bfu|)\in L^1(\Omega)$ be a weak solution to \eqref{system1}, where $H:\Omega\times[0,\infty)\to[0,\infty)$ is defined in \eqref{eq:H} complying with \eqref{eq:pq} and $\bfA$  satisfies \eqref{eq:1.8ok1}--\eqref{eq:degenereassump}. 
We introduce the following \emph{Campanato-type} excess functionals, measuring the oscillations of $D\bfu$:
\begin{equation}
E(x_0,r, {\bf Q}):  = \dashint_{B_{r}(x_0)} |\bfV_{H^-_{B_{r}(x_0)}}(D\bfu)-\bfV_{H^-_{B_{r}(x_0)}}({\bf Q})|^2 \,\mathrm{d} x\,,
\label{eq:excessE}
\end{equation} 
and 
\begin{equation}
\Phi(x_0,r, {\bf Q}) :=\frac{E(x_0,r, {\bf Q})}{H^-_{B_r(x_0)}(|{\bf Q}|)}\,.
\label{eq:excessPhi}
\end{equation}
  If ${\bf Q}=(D\bfu)_{x_0,r}$, we will use the shorthand $E(x_0,r)\equiv E (x_0,r, (D\bfu)_{x_0,r})$ and $ \Phi(x_0,r)\equiv\Phi(x_0,r, (D\bfu)_{x_0,r})$. 
Note that, by \eqref{eq:linkVphi} and \eqref{eq:equivalencebis}, it holds that
\begin{equation}
E(x_0,r) \sim  \dashint_{B_{r}(x_0)} |\bfV_{H^-_{B_{r}(x_0)}}(D\bfu)-(\bfV_{H^-_{B_{r}(x_0)}}(D\bfu))_{x_0,r}|^2 \,\mathrm{d} x\,,
\label{eq:equivalenceexcess}
\end{equation}
where $H_{a} (x,t)$ denotes the shifted $N$ function of $H$ with shift $a$.

Furthermore, we fix any $B_{2r} (x_0)\subset \Omega'\Subset \Omega$ satisfying \eqref{eq:assless1} with $r\le 1/2$. 
We first consider the nondegenerate regime.

\subsection{Non degenerate regime: almost $\mathcal{A}$-harmonic functions} \label{sec:almostAharmonic}

We can start with the linearization procedure for system \eqref{system1}. Let us set
$$
\mathcal{A}({\bf Q}) := \frac{ D_\xi{\bf A}(x^-_{2r},{\bf Q})}{ H''(x^-_{2r}, |{\bf Q}|)} \,,
\quad \bfQ\in \R^{N\times n}
\,.
$$
Note that the bilinear form $\mathcal{A}(\bfQ)$ satisfies the Legendre-Hadamard condition \eqref{eq:LegHam} by virtue of \eqref{eq:1.8ok1}--\eqref{eq:1.8ok2}. 
We aim to prove that the function $\bfu-\bm\ell_{x_0,2r, {\bf Q}}$ is approximately $\mathcal{A}$-harmonic. This fact, together with the higher integrability result \eqref{eq:caccioppoliIbis} will allow us to apply the $\mathcal{A}$-harmonic approximation lemma: Lemma~\ref{thm:Aappr_psi}.
We also note that in this regime we do not use the assumption \eqref{eq:degenereassump}.

\begin{lemma}
Let $\alpha_2$ be the exponent of Lemma \ref{lem:comparison}, and $\beta_0$ be from \eqref{eq:1.12ok}.  
Then there exists $c=c(n,N,p,q,\nu,L,\alpha,[a]_{C^{0,\alpha}})>0$ such that
\begin{equation}
\begin{split}
& \bigg|\dashint_{B_{r}(x_0)} \langle \mathcal{A}({\bf Q}) (D{\bf u}- {\bf Q})\,|\,D\bm \psi \rangle \,\mathrm{d}x\bigg| \\
& \leq c |{\bf Q}|\left\{ \Phi(x_0,2r, {\bf Q}) + [\Phi(x_0,2r, {\bf Q}) ]^\frac{1+\beta_0}{2} +( r^{\alpha_2}+r^{\beta_0})[1+\Phi(x_0,2r, {\bf Q}) ]^\frac{q-1}{p}\right\}\|D\bm \psi\|_{\infty} 
\end{split}
\label{eq:4.12ok}
\end{equation}
for every $\bm \psi\in C^\infty_0(B_{r}(x_0);\R^N)$. 
\label{lem:lemma4.2ok}
\end{lemma}

\begin{proof}
It will suffice to prove \eqref{eq:4.12ok} for $\bm \psi\in C_0^\infty(B_{r}(x_0);\R^N)$ with $\|D\bm \psi\|_{\infty}\leq1$, since the general case will follow by a standard normalization argument. 
To enlighten notation, we omit the explicit dependence on $x_0$, and write $H^{\pm}_{2r}(t) = H^{\pm}_{B_{2r}(x_0)}(t)$.
From the definitions of $\mathcal{A}$ and ${\bf w}$ we have
\begin{equation}
\begin{split}
& H''(x^-_{2r},|{\bf Q}|) \dashint_{B_{r}} \langle  \mathcal{A}({\bf Q})(D{\bf w}-{\bf Q})\,|\,D\bm \psi \rangle  \,\mathrm{d}x  \\
&\,\, = \dashint_{B_{r}} \langle D_\xi{\bf A}(x^-_{2r},{\bf Q}) (D\bfu-{\bf Q}) \,|\, D\bm \psi \rangle \,\mathrm{d}x \\
& \,\, = \dashint_{B_{r}} \int_0^1 \langle  [ D_\xi{\bf A}(x^-_{2r}, {\bf Q}) - D_\xi{\bf A}(x^-_{2r}, {\bf Q} + t (D\bfu-{\bf Q})) ](D\bfu-{\bf Q}) \,|\, D\bm \psi \rangle \,\mathrm{d}t\,\mathrm{d}x \\
& \,\, \,\,\,\, + \dashint_{B_{r}} \int_0^1 \langle  [D_\xi{\bf A}(x^-_{2r}, {\bf Q} + t (D\bfu-{\bf Q})) ] (D\bfu-{\bf Q}) \,|\, D\bm \psi \rangle  \,\mathrm{d}t\,\mathrm{d}x \\
& \,\, =: J_1 + J_2\,.
\end{split}
\label{eq:4.13ok}
\end{equation}
In order to estimate $J_1$, we first observe that
\begin{equation*}
\begin{split}
J_1 & = \dashint_{B_{r}} \mathbbm{1}_{\widetilde E}(x) \int_0^1 \langle [D_\xi{\bf A}(x^-_{2r}, {\bf Q}) - D_\xi{\bf A}(x^-_{2r}, {\bf Q} + t (D\bfu-{\bf Q}))](D\bfu-{\bf Q}) \,|\, D\bm \psi \rangle\,\mathrm{d}t\,\mathrm{d}x \\
& \,\,\,\,\,\, + \dashint_{B_{r}} \mathbbm{1}_{\widetilde F}(x) \int_0^1 \langle [D_\xi{\bf A}(x^-_{2r}, {\bf Q}) - D_\xi{\bf A}(x^-_{2r}, {\bf Q} + t (D\bfu-{\bf Q}))](D\bfu-{\bf Q}) \,|\, D\bm \psi \rangle\,\mathrm{d}t\,\mathrm{d}x \\
& =: J_{1,\widetilde E} + J_{1,\widetilde F}\,,
\end{split}
\end{equation*}
where $\widetilde E :=\left\{x\in B_{r}\,:\, |D\bfu(x)-{\bf Q}|\geq \frac{1}{2}|{\bf Q}|\right\}$, and $\widetilde F :=B_{r}\backslash \widetilde E$.

We start with the estimate of $J_{1,\widetilde E}$. From \eqref{eq:1.8ok1} and \eqref{ineq:phiast_phi_p},
\begin{equation*}
\begin{split}
{|J_{1,\widetilde E}|}& \leq c \dashint_{B_{r}}  \mathbbm{1}_{\widetilde E}(x) \bigg(\int_0^1 \left[ (H^{-}_{{2r}})''(|{\bf Q}|) + (H^{-}_{B_{2r}})''(|{\bf Q} + t (D\bfu-{\bf Q})|)\right]\,\mathrm{d}t \bigg) |D\bfu-{\bf Q}|\,\mathrm{d}x \\
& \lesssim \dashint_{B_{r}}  \mathbbm{1}_{\widetilde E}(x) \left [ (H^{-}_{{2r}})''(|{\bf Q}|) + (H^{-}_{{2r}})''(|{\bf Q}| + |D\bfu|)\right ] |D\bfu-{\bf Q}|\,\mathrm{d}x \\
& \lesssim \dashint_{B_{r}}  \mathbbm{1}_{\widetilde E}(x)  (H^{-}_{{2r}})'(|{\bf Q}| + |D\bfu|)  \frac{|D\bfu-{\bf Q}|}{|{\bf Q}|}\,\mathrm{d}x \,.
\end{split}
\end{equation*}
For a.e. $x\in \widetilde E$, it holds
\begin{equation*}
|{\bf Q}|+|D\bfu| \leq |D\bfu-{\bf Q}| + 2 |{\bf Q}| \leq 5 |D\bfu-{\bf Q}| \,,
\end{equation*}
whence 
\begin{equation*}
(H^{-}_{{2r}})'(|{\bf Q}| + |D\bfu|) \lesssim \frac{(H^{-}_{{2r}})'(|{\bf Q}| + |D\bfu-{\bf Q}|)}{|{\bf Q}| + |D\bfu-{\bf Q}|}  |D\bfu-{\bf Q}| \,.
\end{equation*}
Now, using \eqref{eq:phi_shifted}, \eqref{(2.6a)} and \eqref{ineq:phiast_phi_p} for $\varphi(t):=H^{-}_{{2r}}(t)$, we finally get
\begin{equation*}
|J_{1,\widetilde E}|  \lesssim \frac{1}{|{\bf Q}|}\dashint_{B_{r}} H_{|{\bf Q}|}(x^-_{2r},  |D\bfu-{\bf Q}|)\,\mathrm{d}x   \lesssim |{\bf Q}|(H^{-}_{{2r}})''(|{\bf Q}|) \Phi(2r, {\bf Q})\,.
\end{equation*}
For what concerns $J_{1,\widetilde F}$, {by assumption \eqref{eq:1.12ok} we note that, for every $t\in[0,1]$, 
\begin{equation*}
 \left|D_\xi{\bf A}(x^-_{2r}, {\bf Q}) - D_\xi{\bf A}(x^-_{2r}, {\bf Q} + t (D\bfu-{\bf Q}))\right| \lesssim (H^{-}_{{2r}})''(|{\bf Q}|) \left(\frac{|D\bfu-{\bf Q}|}{|{\bf Q}|}\right)^{\beta_0}\,,
\end{equation*}
so that 
\begin{equation*}
|J_{1,\widetilde F}| \lesssim |{\bf Q}| (H^{-}_{{2r}})''(|{\bf Q}|)\dashint_{B_{r}}  \mathbbm{1}_{\widetilde F}(x) \left(\frac{|D\bfu-{\bf Q}|}{|{\bf Q}|}\right)^{1+\beta_0}\,\mathrm{d}x\,.
\end{equation*} }
For a.e. $x\in \widetilde F$, we have
\begin{equation*}
|{\bf Q}|+|D\bfu-{\bf Q}| < \frac{3}{2}|{\bf Q}|\,,
\end{equation*}
whence, using again \eqref{eq:phi_shifted}, \eqref{(2.6a)} for $\phi=H^{-}_{{2r}}$, we obtain
\begin{equation*}
\begin{split}
\frac{|D\bfu-{\bf Q}|^2}{|{\bf Q}|^2} = \frac{(H^{-}_{{2r}})'(|{\bf Q}|)}{(H^{-}_{{2r}})'(|{\bf Q}|)} \cdot \frac{|D\bfu-{\bf Q}|^2}{|{\bf Q}|^2}& \lesssim \frac{(H^{-}_{{2r}})'(|{\bf Q}| + |D\bfu-{\bf Q}|)|D\bfu-{\bf Q}|^2}{H^{-}_{{2r}}(|{\bf Q}|) (|{\bf Q}| + |D\bfu-{\bf Q}|)}  \\
& \sim \frac{1}{H^{-}_{{2r}}(|{\bf Q}|)} H_{|{\bf Q}|}(x^-_{2r}, |D\bfu-{\bf Q}|) \,.
\end{split}
\end{equation*}
Combining the previous estimates and using Jensen's inequality with $\frac{1+\beta_0}{2}<1$, we get
\begin{equation*}
\begin{split}
|J_{1,\widetilde F}| & \lesssim |{\bf Q}| (H^{-}_{{2r}})''(|{\bf Q}|)\dashint_{B_{r}} \left(\frac{1}{H^{-}_{{2r}}(|{\bf Q}|)} H_{|{\bf Q}|}(x^-_{2r}, |D\bfu-{\bf Q}|) \right)^{\frac{1+\beta_0}{2}}\,\mathrm{d}x \\
& \lesssim |{\bf Q}| (H^{-}_{{2r}})''(|{\bf Q}|) (\Phi(2r, {\bf Q}))^{\frac{1+\beta_0}{2}} \,.
\end{split}
\end{equation*} 
Collecting the estimates for $J_{1,\widetilde E}$ and $J_{1,\widetilde F}$, we then infer
\begin{equation}
|J_1| \lesssim |{\bf Q}| (H^{-}_{{2r}})''(|{\bf Q}|) \left[(\Phi(2r, {\bf Q}))^{\frac{1+\beta_0}{2}} + \Phi(2r, {\bf Q}) \right]\,. 
\label{eq:estimJ1}
\end{equation}

In order to estimate $J_2$, we use Lagrange's Mean Value Theorem, \eqref{eq:4.4ok} combined with the definition of weak solution, \eqref{eq:1.11ok} and we preliminary obtain
\begin{equation}
\begin{split}
|J_2| &  = \bigg|\dashint_{B_{r}}  \langle{\bf A}(x^-_{2r}, D\bfu)-{\bf A}(x, D\bfu) \,|\, D\bm \psi \rangle\,\mathrm{d}x \bigg| \\
& \lesssim r^{\beta_0} \dashint_{B_{r}} H'(x, |D\bfu|) \,\mathrm{d}x +\dashint_{B_{r}}|a(x^-_{2r})- a(x)||D\bfu|^{q-1}\,\mathrm{d}x=: J_3+J_4\,.
\end{split}
\label{eq:J3J4}
\end{equation}
To estimate $J_3$ and $J_4$, we observe that, since on $\widetilde  E$ it holds that $|D\bfu - {\bf Q}| \geq \frac{1}{2} |D\bfu - {\bf Q}| + \frac{1}{4} |{\bf Q}|$, 
\begin{equation}
\begin{split}
|D\bfu|^p & \lesssim \mathbbm{1}_{\widetilde E}(x) |D\bfu - {\bf Q}|^p + |{\bf Q}|^p \\
& \leq  \mathbbm{1}_{\widetilde E}(x) 4\frac{(|D\bfu - {\bf Q}|+|{\bf Q}|)^{p-1}}{|D\bfu - {\bf Q}|+|{\bf Q}|} |D\bfu - {\bf Q}|^2\cdot \frac{|{\bf Q}|^p}{|{\bf Q}|^p}+ |{\bf Q}|^p \\
& \leq  \mathbbm{1}_{\widetilde E}(x) 4\frac{(H^{-}_{{2r}})'(|D\bfu - {\bf Q}|+|{\bf Q}|)}{|D\bfu - {\bf Q}|+|{\bf Q}|} |D\bfu - {\bf Q}|^2\cdot \frac{|{\bf Q}|^{p-1}}{(H^{-}_{{2r}})'(|{\bf Q}|)}+ |{\bf Q}|^p \\
& \lesssim   \mathbbm{1}_{\widetilde E}(x) 4H_{|{\bf Q}|}(x^-_{2r}, |D\bfu - {\bf Q}|)\cdot \frac{|{\bf Q}|^{p-1}}{(H^{-}_{{2r}})'(|{\bf Q}|)}+ |{\bf Q}|^p \\
& \lesssim |{\bf Q}|^p \left( \mathbbm{1}_{\widetilde E}(x) 4H_{|{\bf Q}|}(x^-_{2r}, |D\bfu - {\bf Q}|)\cdot \frac{1}{H^{-}_{{2r}}(|{\bf Q}|)}+ 1 \right)\,,
\end{split}
\label{eq:gradestim}
\end{equation}
where we used \eqref{eq:phi_shifted} and \eqref{ineq:phiast_phi_p} for $\phi=H^{-}_{{2r}}(t)$. 
For $J_3$, using \eqref{eq:H'reverse} with H\"older's inequality, \eqref{eq:gradestim} and \eqref{eq:linkVphi} and \eqref{ineq:phiast_phi_p} for $\phi=H^{-}_{{2r}}$, we have
\begin{equation*}\begin{split}
 J_3
 &\lesssim r^{\beta_0} (H^-_{{2r}})' \bigg(\bigg[\dashint_{B_{2r}}  |D\bfu|^p \,\mathrm{d}x\bigg]^{1/p}\bigg)\\
 & \lesssim r^{\beta_0} (H^-_{{2r}})' \bigg(|\bfQ| \bigg[\frac{1}{H^-_{{2r}}(|\bfQ|)}\dashint_{B_{2r}}H^-_{{2r}} ( |D\bfu- \bfQ|) \,\mathrm{d}x+1\bigg]^{1/p}\bigg)\\
& \lesssim  r^{\beta_0} \big[\Phi(2r,\bfQ)^{\frac{q-1}{p}}+1\big] (H^-_{{2r}})' (|\bfQ| ) \sim r^{\beta_0} \big[\Phi(2r,\bfQ)^{\frac{q-1}{p}}+1\big] (H^-_{{2r}})'' (|\bfQ| ) |\bfQ |\,.
\end{split}\end{equation*}
For $J_4$, by \eqref{eq:comparison1} and the previous estimation, we have
$$\begin{aligned}
J_4  \lesssim  r^{\alpha_2} (H^{-}_{{2r}})' \bigg( \dashint_{B_{2r}} |D\bfu | \,\mathrm{d} x \bigg) \lesssim r^{\alpha_2}  \big[\Phi(2r,\bfQ)^{\frac{q-1}{p}}+1\big] (H^-_{{2r}})'' (|\bfQ| ) |\bfQ |\,,
\end{aligned}$$
so that, taking into account \eqref{eq:J3J4}, we finally get  
\begin{equation}
|J_2| 
\lesssim (r^{\beta_0}+r^{\alpha_2}) \big[\Phi(2r,\bfQ)^{\frac{q-1}{p}}+1\big] (H^-_{{2r}})'' (|\bfQ| ) |{\bf Q}|\,.
\label{eq:estimJ2}
\end{equation}
Therefore, inserting \eqref{eq:estimJ1} and \eqref{eq:estimJ2} into \eqref{eq:4.13ok}, the proof of \eqref{eq:4.12ok} is completed. 
\end{proof}

We now set 
\begin{equation}\label{alpha3}
\alpha_3:=\min\left\{\alpha_0, \alpha_1,\alpha_2,\beta_0 
\right\}, 
\end{equation}
where $\alpha_0$, $\alpha_1$, $\alpha_2$ and  $\beta_0$ are from \eqref{alpha0}, \eqref{alpha1}, Lemma~\ref{lem:comparison} and \eqref{eq:1.12ok}, respectively, and
\begin{equation}
E_*(x_0, \rho)  :  =  E(x_0, \rho)+ \rho^{\frac{\alpha_3}{2}} H^-_{B_\rho(x_0)}(|(D\bfu)_{x_0, B_{\rho}(x_0)}|)  = H^-_{B_\rho(x_0)}(|(D\bfu)_{x_0,B_{\rho}(x_0)}|)\left(\Phi(x_0, \rho)+ \rho^{\frac{\alpha_3}{2}}\right)\,,
\label{eq:excess2}
\end{equation}
where the excess $E(x_0, \rho)$ was introduced in \eqref{eq:excessE}. For the ease of reading, we also recall the definition of $\bfV_{H^-_{B_\rho(x_0)}}$ given in \eqref{eq:defV}; namely,
\begin{equation*}
\bfV_{H^-_{B_\rho(x_0)}}(\bfP) =  \sqrt{\frac{(H^-_{B_\rho(x_0)})'(|\bfP|)}{|\bfP|}} \bfP\,, \quad \bfP\in\R^{N\times n}\,.
\end{equation*}
Now, we can prove the excess decay estimate in the non-degenerate regime. 

\begin{lemma}
For every $\epsilon>0$, there exist small $\delta_1,\delta_2\in (0,1)$ depending on $n$, $N$, $p$, $q$, $\nu$, $L$, $\alpha$, $[a]_{C^{0,\alpha}}$, $\beta_0$ and $\epsilon$  such that if
\begin{equation}\label{smallness1}
 \dashint_{B_{2r}(x_0)} \left|\bfV_{H^-_{B_{2r}(x_0)}}(D\bfu)-\left(\bfV_{H^-_{B_{2r}(x_0)}}(D\bfu)\right)_{x_0,2r} \right|^2 \,\mathrm{d} x\,\le  \delta_1 \dashint_{B_{2r}(x_0)} \left|\bfV_{H^-_{B_{2r}(x_0)}}(D\bfu)\right|^2 \,\mathrm{d} x 
\end{equation}
and
\begin{equation}\label{smallness2}
r^{\frac{\alpha_3}{2}} \le \delta_2\,,
\end{equation}
then for every $\tau\in(0,1/4)$
$$
 \dashint_{B_{\tau r}(x_0)} \left|\bfV_{H^-_{B_{\tau r}(x_0)}}(D\bfu)-\left(\bfV_{H^-_{B_{\tau r}(x_0)}}(D\bfu)\right)_{x_0,\tau r}\right|^2 \,\mathrm{d} x 
 \le  c \tau^{2}\left(1+\frac{\epsilon}{\tau^{n+2}} \right)E_*(x_0,2r)\,.
$$
\label{lem:nondegeneratedecay}
\end{lemma}
\begin{proof}
In order to enlighten the notation, we will omit the dependence on $x_0$ and write $\bfV_{H^\pm_{2r}}$ and $H^\pm_{2r}$ in place of $\bfV_{H^\pm_{B_{2r}(x_0)}}$ and $H^\pm_{B_{2r}(x_0)}$, respectively. Set ${\bf Q}=(D\bfu)_{2r}$
 and denote by ${\bm \ell}_{2r}:=\bm\ell_{x_0,2r,(D\bfu)_{x_0 , 2r}}$. 
We first observe from \eqref{eq:equivalencebis} and \eqref{smallness1} that
$$\begin{aligned}
\dashint_{B_{2r}} \left|\bfV_{H^-_{2r}}(D\bfu)\right|^2 \,\mathrm{d} x  
& \le 2  \dashint_{B_{2r}} \left|\bfV_{H^-_{2r}}(D\bfu)-\bfV_{H^-_{2r}}(\bfQ)\right|^2 \,\mathrm{d} x +2 \left|\bfV_{H^-_{2r}}(\bfQ)\right|^2\\
& \le c   \dashint_{B_{\tau r}} \left|\bfV_{H^-_{2 r}}(D\bfu)-\big(\bfV_{H^-_{2 r}}(D\bfu) \big)_{2r}\right|^2 \,\mathrm{d} x  +2 \left|\bfV_{H^-_{2r}}(\bfQ)\right|^2\\
& \le c  \delta_1 \dashint_{B_{2r}} \left|\bfV_{H^-_{2r}}(D\bfu)\right|^2 \,\mathrm{d} x   +2 \left|\bfV_{H^-_{2r}}(\bfQ)\right|^2\,.
\end{aligned}$$
We choose $\delta_1\in (0,1)$ small, so that $ c  \delta_1 \le 1/2$,
hence, using the definition of $\bfV_{H^-_{2r}}$ and the fact that $|\bfV_{H^-_{2r}}({\bf P}) |^2 \sim H^-_{2r}(|{\bf P}|)$, we obtain 
\begin{equation}\label{eq:estimateDuQ}
\dashint_{B_{2r}} H^-_{2r}(|D\bfu|) \,\mathrm{d} x   \le c H^-_{2r}(|\bfQ|)\,. 
\end{equation} 

Using \eqref{eq:4.12ok}, \eqref{eq:excessPhi}, \eqref{eq:excess2} and the fact that $\Phi(2r) \lesssim \delta \le 1$ by  \eqref{eq:equivalenceexcess} and \eqref{smallness1}, we get
\begin{equation}\label{eq:4:A-harmonic-inequality}
\begin{split}
\bigg|\dashint_{B_{r}}\langle\mathcal{A}({\bf Q}) (D \bfu - \bfQ) \,|\, D\bm \psi\rangle \,\mathrm{d}x\bigg| 
&\leq c \left\{ \Phi(2r)^{\frac{1}{2}} + \Phi(2r)^\frac{\beta_0}{2} + r^{\frac{\alpha_3}{2}} \right\} \left(\frac{E_*(2r)}{H^-_{2r}(|\bfQ|)}\right)^{\frac{1}{2}} |{\bf Q}| \|D\bm \psi\|_{L^\infty}\\
& \leq \tilde c_{1} \left\{ \delta_1^{\frac{1}{2}} + \delta_1^\frac{\beta_0}{2} + \delta_2 \right\} \left(\frac{E_*(2r)}{H^-_{2r}(|\bfQ|)}\right)^{\frac{1}{2}} |{\bf Q}| \|D\bm \psi\|_{L^\infty}
\end{split}
\end{equation}
for every $\bm \psi\in C^\infty_0(B_{r};\R^N)$.

We next define an $N$-function $\zeta$ by 
\begin{equation} 
\label{eq:psi-def}
   \zeta(t) := \frac{(H^-_{2r})_{|\bfQ|}(t)}{H^-_{2r}(|\bfQ|)}  \sim \frac{H^-_{2r}(|\bfQ|+t)}{H^-_{2r}(|\bfQ|) (|\bfQ|+t)^2}t^2\,, \quad t\ge 0,
   \end{equation} 
where  the equivalence follows by \eqref{(2.6b)}. 
Then  we have 
\begin{displaymath} 
     \left(\frac{t}{|\bfQ|}\right)^2  \leq 4  \frac {H^-_{2r}(|\bfQ| + t)}{H^-_{2r}(|\bfQ|)(|\bfQ|+ t)^2}t^2 
\le  \tilde c_2 \zeta(t), \qquad 
t \in [0,|\bfQ|],   
\end{displaymath}  
for some $\tilde c_2\ge 1$. Moreover, we observe from Lemma~\ref{cor:reverse} and \eqref{eq:linkVphi} that
\begin{align*} 
       \bigg( 
          \dashint_{B_{r}} 
             [\zeta( |D\bfu-\bfQ|)]^{1 + \sigma} \,\d x 
       \bigg)^{\frac {1}{1 + \sigma}}  
& =    \frac {1}{H^-_{2r}(|\bfQ|)}  \bigg( \dashint_{B_{r}} (H^-_{2r})_{|\bfQ|}(|D\bfu-\bfQ|)^{1 + \sigma} \, \d x  \bigg)^{\frac {1}{1 + \sigma}} \\
& \leq \frac {c}{H^-_{2r}(|\bfQ|)}  \dashint_{B_{2r}}  (H^-_{2r})_{|\bfQ|}(|D\bfu - \bfQ |) \, \d x + c  r^{\alpha_1}  \\
& \leq \tilde{c}_3 \, \frac {E_*(2r)}{H^-_{2r}( |\bfQ|)} 
\end{align*}
holds for some constant $\tilde{c}_3 \ge 1$.  

With the constants $\tilde c_1,\tilde c_2,\tilde c_3\ge 1$ determined above and $\tilde c_5\ge 1$ determined below, we define 
$$
\mu
:=   \max \left\{ \tilde{c}_1,\sqrt{\tilde{c}_2\tilde{c}_3(2\tilde{c}_5)^{1/p}} \right\} 
     \left[ \frac {E_*(2r)}{H^-_{2r}( |\bfQ|)} \right]^{\frac 12} |\bfQ|\,.
$$
Then, choosing $\delta_i$ ($i=1,2$) sufficiently small, we see that 
\begin{equation} 
\label{eq:4:mu}
 \mu \leq  \max \left\{ \tilde{c}_1,\sqrt{\tilde{c}_2\tilde{c}_3(2\tilde{c}_5)^{1/p}} \right\} 
     ( q\delta_1 + \delta_2)^{\frac 12}|\bfQ|
<    |\bfQ|.  
\end{equation}
Combining the previous estimates, we obtain 
\begin{equation} 
\label{eq:4:H(Dv)}
\begin{aligned} 
     \bigg(  \dashint_{B_{r}}  [\zeta( | D\bfu - \bfQ|)]^{1 + \sigma} \mathrm{d}x 
     \bigg)^{\frac {1}{1 + \sigma}}                                                    
\leq \tilde{c}_3 \, \frac {E_*(2r)}{H^-_{2r}( |\bfQ|)}                              
\leq \frac{1}{\tilde{c}_2 (2\tilde{c}_5)^{1/p}}\left(\frac{\mu}{|\bfQ|}\right)^2
\leq \frac{1}{(2 \tilde{c}_5)^{1/p}} \zeta(\mu) \,.
\end{aligned}
\end{equation}  
For given $\epsilon$ and $\zeta$ defined as above, we determine the constant $\delta$ as the one in  Lemma~\ref{thm:Aappr_psi}. Then choosing $\delta_i$ ($i=1,2$) sufficiently small such that
\begin{equation} \delta_1^{\frac12} 
                + 
                \delta_1^{\frac{\beta_0}{2}} 
                + 
                \delta_2 \le \delta
\label{smallness}
\end{equation}
 and inserting \eqref{eq:4:mu} and \eqref{smallness} into \eqref{eq:4:A-harmonic-inequality}, we obtain
$$  
\dashint_{B_{r}}  \langle \mathcal{A}(D\bfu - \bfQ) \,|\, D {\bm \psi} \rangle \, \d x                                                                   
\leq  \frac {\tilde{c}_1 ( \delta_1^{\frac12} +  \delta_1^{\frac{\beta_0}{2}} + \delta_2  )}{ \max \left\{ \tilde{c}_1,\sqrt{\tilde{c}_2\tilde{c}_3(2\tilde{c}_5)^{1/p}} \right\} 
} \mu \|D{\bm \psi}\|_{\infty} 
            \le \delta \mu \|D{\bm \psi}\|_{\infty}.
$$
Therefore, we can apply 
Lemma~\ref{thm:Aappr_psi} to the function $\bfu-{\bm \ell}_{2r}$ in place of $\bfu$ and $\phi=\zeta$, so that recalling the definition of $\zeta$ in \eqref{eq:psi-def} we have 
$$
       \frac{1}{H^-_{2r}(|\bfQ|)} 
       \dashint_{B_{r}} (H^-_{2r})_{|\bfQ|}(|D\bfu-\bfQ-D\bfw|)  \, \d x  \leq \varepsilon \zeta(\mu)\,, 
$$
where $\bfw$ is the $\mathcal{A}$-harmonic function in $B_{r}$ with $\bfw = \bfu - {\bm \ell}_{2r}$
on $\partial B_{r}$. Moreover, since
\begin{displaymath} 
          \zeta(\mu) 
\leq c \frac {H^-_{2r}(|\bfQ| + \mu)}{H^-_{2r}(|\bfQ|)(|\bfQ| + \mu)^2}   \mu^2
\leq c \left(\frac{\mu}{|\bfQ|}\right)^2\le c \frac {E_*(2r)}{H^-_{2r}( |\bfQ|)}
\end{displaymath} 
by \eqref{eq:4:mu}, we obtain
\begin{equation} 
\label{eq:4:A-harmonic-mu}
    \dashint_{B_{r}} (H^-_{2r})_{|\bfQ|}(|D\bfu-\bfQ-D\bfw|)  \, \mathrm{d}x  \leq  \tilde{c}_4  \varepsilon E_*(2r)
\end{equation}
for a suitable constant $\tilde{c}_4 > 0$.  We further notice from the gradient estimates for $\bfw$ in \eqref{eq:westimate1} and \eqref{eq:westimate2} and Jensen's inequality that 
\begin{equation*}
\sup_{B_{r/2}}|D\bfw|  \le c \zeta^{-1} \bigg(\dashint_{B_r}\zeta (|D\bfw|) \, \d x \bigg) \le  \tilde{c}_5\zeta^{-1} \bigg(\dashint_{B_r}\zeta (|D\bfu-\bfQ|) \, \d x \bigg)
\end{equation*}
for some $\tilde{c}_5\ge 1$. Therefore, by \eqref{eq:4:H(Dv)} and \eqref{eq:4:mu}, we see that 
\begin{equation}\label{eq:supDwQ}
\sup_{B_{r/2}}|D\bfw|  \le \frac{1}{2} \mu \le \frac{1}{2}|\bfQ|.
\end{equation}

Fix $\tau \in (0,1/4)$. Note that the previous estimate yields  $\frac12 |\bfQ| \le |\bfQ + (D\bfw)_{\tau r}|\le \frac{3}{2}|\bfQ|$, from which, using also \eqref{eq:linkVphi}, we have 
\begin{displaymath} 
\begin{aligned}
&\dashint_{B_{\tau r}}  \Big\vert \bfV_{H^-_{\tau r}}(D\bfu) - \left(\bfV_{H^-_{\tau r}}(D\bfu)\right)_{\tau r} \Big\vert ^2 \, \d x \\
& \le \dashint_{B_{\tau r}} \Big\vert \bfV_{H^-_{\tau r}}(D\bfu) - \bfV_{H^-_{\tau r}}(\bfQ +   (D\bfw)_{\tau r}) \Big\vert^2 \, \d x\\
& \lesssim \dashint_{B_{\tau r}}(H^-_{\tau r})_{|\bfQ +   (D\bfw)_{\tau r}|} (|D\bfu - \bfQ -   (D\bfw)_{\tau r}|) \, \d x\\
& \sim  \dashint_{B_{\tau r}}(H^-_{\tau r})_{|\bfQ|} (|D\bfu - \bfQ -   (D\bfw)_{\tau r}|) \, \d x\\
         &\lesssim \dashint_{B_{\tau r}} \left[(H^-_{\tau r})_{|\bfQ|} (|D\bfu - \bfQ -   D\bfw|) - (H^-_{2r})_{|\bfQ|} (|D\bfu - \bfQ -   D\bfw|) \right]\, \d x \\
&\qquad + \dashint_{B_{\tau r}}  (H^-_{2r})_{|\bfQ|} (|D\bfu - \bfQ -   D\bfw|) \, \d x +  \dashint_{B_{\tau r}} (H^-_{\tau r})_{|\bfQ|} (|D\bfw -   (D\bfw)_{\tau r}|) \, \d x           \\
& =: I_1+I_2+I_3\,.
\end{aligned}
\end{displaymath}
We estimate $I_1$, $I_2$ and $I_3$, separately. Note that by \eqref{eq:4:A-harmonic-mu},
$$
I_2 \lesssim  \varepsilon \tau^{-n} E_*(2r) \,.
$$
 For $I_1$, using the gradient estimates for $\bfw$ in \eqref{eq:westimate1} and \eqref{eq:westimate2} with $\psi(t)=t^p$, H\"older's inequality, \eqref{eq:comparison2}, \eqref{eq:estimateDuQ}, and the smallness assumption \eqref{smallness2} with choosing $\delta_2\le \epsilon$, we have
$$\begin{aligned}
I_1 & \lesssim \dashint_{B_{\tau r}}( a(x)-a^-_{{2r}} ) (|D\bfu|^q + |(D\bfu)_{2r}|^q + |D\bfw|^{q})\, \d x\\
& \lesssim \tau^{-n} \dashint_{B_{r}}( a(x)-a^-_{{2r}} ) |D\bfu|^q\, \d x+r^\alpha |(D\bfu)_{2r}|^q +(|D\bfu|^p)_{r}^{\frac{q}{p}} \dashint_{B_{\tau r}}( a(x)-a^-_{{2r}} ) \, \d x  \\
& \lesssim \tau^{-n} \dashint_{B_{r}}( a(x)-a^-_{{2r}} ) |D\bfu|^q\, \d x+  r^\alpha  (|D\bfu|^{p(1+\sigma_0)})_{2r}^{\frac{q-p}{p(1+\sigma_0)}} (H^-_{2r}(|D\bfu|))_{2r} \\
& \lesssim  r^{\alpha_2} \tau^{-n}  (H_{2r}^-( |D\bfu| ))_{B_{2r}}  +   r^{\alpha_0}    (H^-_{2r}(|D\bfu|))_{2r}\\
& \lesssim   r^{\alpha_3} \tau^{-n}  H^-_{2r}(|\bfQ|) \lesssim \varepsilon \tau^{-n} E_*(2r) \,,
\end{aligned}$$ 
where we used also estimate \eqref{eq:estimq-p} and the definition of $\alpha_3$ in  \eqref{alpha3}. 
For $I_3$, by \eqref{(2.6b)}, the regularity estimates for $\bfw$ in \eqref{eq:westimate1}, \eqref{eq:supDwQ} and \eqref{eq:westimate2} with $\phi(t)=t^p$
$$\begin{aligned}
I_3   \lesssim   (H^-_{\tau r})_{|\bfQ|} (\tau r \sup_{B_{r/4}}\ |D^2\bfw|)   &\lesssim    (H^-_{\tau r})_{|\bfQ|} (\tau   \sup_{B_{r/2}}  |D\bfw|)  
\sim \tau^2  (H^-_{\tau r})_{|\bfQ|} \bigg(  \sup_{B_{r/2}}\, |D\bfw| \bigg)  \\
&   \lesssim \tau^2  (H^+_{2r})_{|\bfQ|} \Big(  (|D\bfw|^p)_{B_{r}}^{1/p} \Big) 
\lesssim \tau^2  H^+_{2r} \Big(  (|D\bfu -{\bf Q}|^p)_{B_{r}}^{1/p} \Big) \,.
\end{aligned}$$
Moreover, by \eqref{eq:estimq-p} and \eqref{eq:estimateDuQ}, we have
\begin{equation*}
I_3    \lesssim  \tau^{2} \left[ (|D\bfu - {\bf Q}|^p)_{B_{2r}} + a^-_{2r} (|D\bfu - {\bf Q}|^p)_{B_{2r}}^{q/p}  + r^{\alpha}(|D\bfu - {\bf Q}|^p)_{B_{2r}}^{(q-p)/p} \right]  \lesssim  \tau^{2} r^{\alpha_3} H^-_{2r}(|\bfQ|) \lesssim \tau^{2} E_*(2r) \,.
\end{equation*}
Consequently, combining the above results, we obtain the desired estimate.  
\end{proof}

\subsection{Degenerate regime: almost $\varphi$-harmonic functions} \label{sec:almostphiharmonic}

Now, we deal with the degenerate regime. Here, we use the assumption \eqref{eq:degenereassump}, in place of \eqref{eq:1.12ok}. 

Fix $B_{2r}=B_{2r}(x_0)\subset \Omega'$, for some $\Omega'\Subset \Omega$, satisfying \eqref{eq:assless1}. 
We further introduce the \emph{Morrey-type} excess
\begin{equation}
\Psi(x_0, \rho):=\dashint_{B_\rho(x_0)} H^{-}_{B_{\rho}(x_0)}(|D\bfu|)\,\mathrm{d}x \,. 
\label{eq:morreyexcess}
\end{equation}

The first result is that every weak solution to \eqref{eq:system1} is almost $H^{-}_{B_{2r}}$-harmonic.

\begin{lemma}
For every $\delta\in(0,1)$, the inequality
\begin{equation}\label{lem:almostHharmonic}
\begin{split}
& \left |\dashint_{B_{r}(x_0)} \bigg\langle (H_{B_{2r}(x_0)}^-) '(|D \bfu|)\frac{D \bfu}{|D \bfu|} \, \bigg|\, D\bm \psi \bigg\rangle\,\mathrm{d}x\right| \\
& \,\,\,\, \leq c_* \left( \delta+ \frac{(H_{B_{2r}(x_0)}^-)^{-1}(\Psi(x_0, 2r))}{\kappa} + r^{\alpha_2}\right)\left(\Psi(x_0,2r) + H_{B_{2r}(x_0)}^-(\|D\bm \psi\|_\infty)\right)
\end{split}
\end{equation}
holds for every $\bm \psi\in C^\infty_0(B_{r}(x_0);\R^n)$ and for some constant $c_*=c_*(n,N,p,q,\nu,L, \alpha, [a]_{C^{0,\alpha}})>0$, where $\kappa=\kappa(\delta)>0$ is given in \eqref{eq:degenereassump}  and $\alpha_2$ is the exponent in  Lemma \ref{lem:comparison}.
\end{lemma}

\begin{proof}
For simplicity, we write $B_\rho=B_\rho(x_0)$ and $H_{\rho}^-(t):=H_{B_{\rho}(x_0)}^-(t)$, and $\Psi(\rho)=\Psi(x_0,\rho)$  for $\rho\in(0,2r]$. Let $\bm \psi\in C^\infty_0(B_{r};\R^n)$ be such that $\|D\bm \psi\|_\infty\leq1$. Then, by the definition of weak solution, we have 
\begin{equation*}
\begin{split}
\dashint_{B_{r}} \bigg\langle (H_{2r}^-) '(|D \bfu|)\frac{D \bfu}{|D \bfu|} \,\bigg| \, D\bm \psi \bigg\rangle \,\mathrm{d}x & = \dashint_{B_{r}} \bigg\langle (H_{2r}^-)'(|D \bfu|)\frac{D \bfu}{|D \bfu|}-H'(x,|D \bfu|)\frac{D \bfu}{|D \bfu|}  \,\bigg| \, D\bm \psi \bigg\rangle\,\mathrm{d}x \\
& \,\,\,\,\,\, + \dashint_{B_{r}} \bigg\langle H'(x,|D \bfu|)\frac{D \bfu}{|D \bfu|}-{\bf A}(x, D\bfu)\,\bigg|\,  D\bm \psi \bigg\rangle\,\mathrm{d}x \\
& =: I_1+I_2\,. 
\end{split}
\end{equation*}
We start with the estimate of $I_2$. Observe that
$$\begin{aligned}
|I_2| & \le \delta \dashint_{B_r} H'(x,|D\bfu|)\chi_{\{|Du|\le \kappa \}} \, \mathrm{d}x + c \dashint_{B_r} H'(x,|D\bfu|) \chi_{\{|Du| > \kappa \}} \, \mathrm{d}x\\
&  \le \delta \dashint_{B_r} H'(x,|D\bfu|) \, \mathrm{d} x + \frac{c}{\kappa} \dashint_{B_r} H(x,|D\bfu|)  \, \mathrm{d} x\,,
\end{aligned}$$
where we used \eqref{eq:degenereassump} on the set $\{|Du| \leq \kappa \}$, while we exploited the growth assumption \eqref{eq:1.8ok1} combined with \eqref{ineq:phiast_phi_p} for $\varphi(t):=H(x,t)$, for every fixed $x$, elsewhere. 
Note that  from \eqref{eq:Hreverse0}, \eqref{eq:H'reverse}  and Jensen's inequality with \eqref{eq:morreyexcess},
$$
\dashint_{B_r} H(x,|D\bfu|) \, \mathrm{d} x  \le c  H_{2r}^-\left((D{\bf u})_{2r}\right)\le c  (H_{2r}^-)^{-1}\left(\Psi(2r)\right)  ((H_{2r}^-)' \circ (H_{2r}^-)^{-1})\left(\Psi(2r)\right)  \,,
$$
and
$$
\dashint_{B_r} H'(x,|D\bfu|) \, \mathrm{d} x \le  c  ( (H_{{2r}}^-)' \circ (H_{{2r}}^-)^{-1}) (\Psi(2r))\,.
$$
Therefore, we have 
$$
|I_2| \le c \left(\delta+ \frac{(H_{2r}^-)^{-1}(\Psi(2r))}{\kappa}\right)  ((H_{2r}^-)'\circ (H_{2r}^-)^{-1})\left(\Psi(2r)\right).
$$
Moreover, we have from \eqref{eq:comparison1} that
$$
|I_1| \le c r^{\alpha_2} ((H_{2r}^-)'\circ (H_{2r}^-)^{-1})\left(\Psi(2r)\right).
$$

Collecting all the previous estimates, we obtain 
\begin{equation*}
\begin{split}
 & \bigg |\dashint_{B_{r}} (H_{2r}^-) '(|D \bfu|)\frac{D \bfu}{|D \bfu|}: D\bm \psi\,\mathrm{d}x\bigg| \\
 &\hspace{1.5cm} \lesssim \left(\delta+ \frac{(H_{2r}^-)^{-1}(\Psi(2r))}{\kappa} + r^{\alpha_2} \right)((H_{2r}^-)'\circ (H_{2r}^-)^{-1})\left(\Psi(2r)\right) \|D\bm \psi\|_\infty\,.
\end{split}
\end{equation*}
To conclude, we use \eqref{eq:2.6ok} and Young's inequality, to obtain
\begin{equation*}
\begin{split}
 ((H_{B_{2r}}^-)'\circ (H_{B_{2r}}^-)^{-1})\left(\Psi(2r)\right) \|D\bm \psi\|_\infty
 &\leq c (H_{2r}^-)^*\left(((H_{2r}^-)'\circ (H_{2r}^-)^{-1})(\Psi(2r))\right) + H_{2r}^-(\|D\bm \psi\|_\infty) \\
&\leq c \Psi(2r) + H_{2r}^-(\|D\bm \psi\|_\infty)\,.
\end{split}
\end{equation*}
\end{proof}

%
%

We recall the exponent $\gamma_0\in(0,1)$ from Proposition~\ref{lemma:2:G-harmonic-holder}.  We are now in position to prove the excess decay estimate in the degenerate regime. 
\begin{lemma}\label{lemma:4:degenerate-decay}
For every $\gamma \in (0,\gamma_0)$ and $\chi>0$, there exists $\tau,\delta_3,\delta_4>0$ depending on $n$, $N$, $p$, $q$, $\nu$, $L$, $\alpha$, $[a]_{C^{0,\alpha}}$, $\gamma$ and $\chi$ such that if
\begin{equation}
\chi  \dashint_{B_{2r}(x_0)} \left|\bfV_{H^-_{2r}}(D\bfu)\right|^2 \,\mathrm{d} x 
\le  \dashint_{B_{2r}(x_0)} \left|\bfV_{H^-_{2r}}(D\bfu)-\left(\bfV_{H^-_{2r}}(D\bfu)\right)_{x_0,2r}\right|^2 \,\mathrm{d} x\,,
\label{eq:smallexcess1}
\end{equation}
\begin{equation}
\dashint_{B_{2r}(x_0)} \left|\bfV_{H^-_{2r}}(D\bfu)-\left(\bfV_{H^-_{2r}}(D\bfu)\right)_{x_0,2r}\right|^2 \,\mathrm{d} x \le \delta_3\,,
\label{eq:smallexcess2}
\end{equation}
and
\begin{equation}
r^{\alpha_2} \le \delta_4\,,
\label{eq:smallradius}
\end{equation}
then
\begin{equation}
\begin{split}
 \dashint_{B_{2\tau r}(x_0)} & \left |\bfV_{H^-_{\tau r}}(D\bfu)-\left(\bfV_{H^-_{\tau r}}(D\bfu)\right)_{x_0,\tau r}\right|^2 \,\mathrm{d} x \\
 & \qquad \qquad \,\,\,\,\,\,\,\,\,\,\,\,\,\,\,\,\,\,\,\,\,\,\, \,\,\,\,\,\,\, \le  \tau^{2\gamma} \dashint_{B_{2r}(x_0)} \left|\bfV_{H^-_{2r}}(D\bfu)-\left(\bfV_{H^-_{2r}}(D\bfu)\right)_{x_0,2r}\right|^2 \,\mathrm{d} x\,.
\end{split}
\label{eq:degeneratedecayest}
\end{equation}
Here, $H^-_{2r}:=H^-_{B_{2r}(x_0)}$.
\end{lemma}

\begin{proof}
To enlighten notation, we omit the explicit dependence $x_0$. 

We first determine $\tau=\tau(n,N,p,q,\nu,L,\alpha,[a]_{C^{0,\alpha}},\gamma,\chi)>0$ small so that
\begin{equation}\label{choosetau}
\tau \le \frac{1}{4} 
\quad\text{and}\quad
\tilde c_6 \tau^{2\gamma_0}\chi^{-1} \le \tau^{2\gamma}\,, 
\end{equation}
where $\tilde c_6=\tilde c_6(n,N,p,q,\nu,L,\alpha,[a]_{C^{0,\alpha}})>0$ will be determined later. Set 
\begin{equation}\label{chooseepsilon}
\epsilon = \tau^{2\gamma_0+n} \,.
\end{equation}
For this $\epsilon$, we denote the constant $\delta$  in Lemma~\ref{lem:phiharmapprox}, when $\phi=H^-_{2r}$, $s=1+\sigma$ ($\sigma$ is the constant determined in Lemma~\ref{lem:high}), and $c_0$ is the constant $c$ given in Lemma~\ref{cor:reverse}, by  $\delta_0$. We then choose $\delta$ such that
$$
c_* \delta \le \frac{\delta_0}{2}\,,
$$
where $c_*$ denotes the constant in  \eqref{lem:almostHharmonic}, and  hence $\kappa=\kappa(\delta)$ in \eqref{eq:degenereassump} is also determined. Moreover, by the first two assumptions \eqref{eq:smallexcess1} and \eqref{eq:smallexcess2} we have
$$
\Psi(2r) \le c  \dashint_{B_{2r}} \left|\bfV_{H^-_{2r}}(D\bfu)\right|^2 \,\mathrm{d} x  \le \frac{c}{\chi} \dashint_{B_{2r}} \left|\bfV_{H^-_{2r}}(D\bfu)-\left(\bfV_{H^-_{2r}}(D\bfu)\right)_{2r}\right|^2 \,\mathrm{d} x \le \frac{c}{\chi}\delta_3\,.
$$
Hence, with \eqref{eq:smallradius}, we have
$$
c_* \left( \delta+ \frac{(H_{2r}^-)^{-1}(\Psi(2r))}{\kappa} + r^{\alpha_2}\right) \le \frac{\delta_0}{2}+ c_*\max\{\chi^{-\frac{1}{p}},\chi^{-\frac{1}{q}}\}\frac{(H^-_{2r})^{-1}(\delta_1)}{\kappa} + c_* \delta_4\,.
$$
We choose $\delta_3$ and $\delta_4$ such that 
$$
c_*\max\{\chi^{-\frac{1}{p}},\chi^{-\frac{1}{q}}\}\frac{(H^-_{2r})^{-1}(\delta_4)}{\kappa} + c_* \delta_4 \le  \frac{\delta_0}{2}.
$$ 
Therefore, by Lemma~\ref{lem:phiharmapprox} with $\phi=H^-_{2r}$, we have
\begin{equation}\label{Hharmonicapproximation}
\dashint_{B_{r}} \left|\bfV_{H^-_{
2r}}(D\bfu)-\bfV_{H^-_{2r}}(D\bfw)\right|^2 \,\mathrm{d} x \le  c \epsilon  \Psi(2r),
\end{equation}
where $\bfw \in \bfu+W^{1,H^-_{2r}}_0(B_r)$ is the unique $H^-_{2r}$-harmonic mapping coinciding with $\bfu$ on $\partial B_r$.

Therefore, for $\tau\in (0,1/4)$,
\begin{equation}\begin{split}\label{eq:degenerateexcesstau}
\dashint_{B_{2\tau r}} \biggl|\bfV_{H^-_{2\tau r}}&(D\bfu)-\left(\bfV_{H^-_{2\tau r}}(D\bfu)\right)_{2\tau r}\biggr|^2 \,\mathrm{d} x  \le   \dashint_{B_{2\tau r}} \biggl|\bfV_{H^-_{2\tau r}}(D\bfu)-\left(\bfV_{H^-_{2r}}(D\bfw)\right)_{2\tau r}\biggr|^2 \,\mathrm{d} x\\
& \le 4 \dashint_{B_{2\tau r}} \biggl|\bfV_{H^-_{2\tau r}}(D\bfu)-\bfV_{H^-_{2r}}(D\bfu)\biggr|^2 \,\mathrm{d} x  + 4 \dashint_{B_{2\tau r}} \biggl|\bfV_{H^-_{2r}}(D\bfu)-\bfV_{H^-_{2r}}(D\bfw)\biggr|^2 \,\mathrm{d} x\\
&\qquad + 2 \dashint_{B_{2\tau r}} \biggl|\bfV_{H^-_{2r}}(D\bfw)-\left(\bfV_{H^-_{2r}}(D\bfw)\right)_{2\tau r}\biggr|^2 \,\mathrm{d} x\,.
\end{split}\end{equation}
Note that, since $|\sqrt{1+t_1}-\sqrt{1+t_2}|^2\le |t_1-t_2|$ for $t_1,t_2\ge 0$, for every ${\bf P}\in\R^{N\times n}$ we have 
\begin{equation*}
\begin{split}
\left|\bfV_{H^-_{2\tau r}}(\bfP) - \bfV_{H^-_{2r}}(\bfP) \right|^2 & = |\bfP|^{p}\left|\sqrt{1+a^-_{2\tau r}\frac{q}{p}|\bfP|^{q-p}} -\sqrt{1+a^-_{2r}\frac{q}{p}|\bfP|^{q-p}} \right|^2\\
& \le \frac{q}{p} (a^-_{2\tau r}-a^-_{2r}) |\bfP|^q\,.
\end{split}\end{equation*}
Using this, \eqref{eq:comparison2}  and \eqref{eq:smallradius}, we have
\begin{equation}\label{eq:VHcomparison2}\begin{split}
 \dashint_{B_{2\tau r}} \left|\bfV_{H^-_{2\tau r}}(D\bfu)-\bfV_{H^-_{2r}}(D\bfu)\right|^2 \,\mathrm{d} x 
 & \le c  \dashint_{B_{2\tau r}}  (a^-_{2\tau r}-a^-_{2r}) |D\bfu|^q\,\mathrm{d} x\\ 
 &\le c  \tau^{-n} \dashint_{B_{r}}  (a(x)-a^-_{2r}) |D\bfu|^q
 \,\mathrm{d} x\\
 & \le c \tau^{-n} r^{\alpha_2}\dashint_{B_{2r}} \left|\bfV_{H^-_{2r}}(D\bfu)\right|^2 \,\mathrm{d} x \\
 & \le \tilde c_6 \tau^{-n} \delta_4 \dashint_{B_{2r}} \left |\bfV_{H^-_{2r}}(D\bfu)\right|^2 \,\mathrm{d} x
 \end{split}\end{equation}
 for some constant $\tilde c_6>0$. We further choose $\delta_4$ such that $ \delta_4\le \tau^{n+2\gamma_0}$.
Inserting \eqref{eq:excessdecayw}, \eqref{Hharmonicapproximation} with \eqref{chooseepsilon}, and \eqref{eq:VHcomparison2} into \eqref{eq:degenerateexcesstau} and using  assumption \eqref{eq:smallexcess1} and \eqref{choosetau}, we finally obtain
\begin{equation*}\begin{split}
\dashint_{B_{2\tau r}} \left|\bfV_{H^-_{2 \tau r}}(D\bfu)-\left(\bfV_{H^-_{\tau r}}(D\bfu)\right)_{2 \tau r}\right|^2 \,\mathrm{d} x 
&\le \tilde c_1 \tau^{2\gamma_0}  \dashint_{B_{2r}} \left|\bfV_{H^-_{2r}}(D\bfu)\right|^2 \,\mathrm{d} x\\
& \le \tilde c_1 \tau^{2\gamma_0} \chi^{-1} \dashint_{B_{2r}} \left|\bfV_{H^-_{2r}}(D\bfu)-\left(\bfV_{H^-_{2r}}(D\bfu)\right)_{2r}\right|^2 \,\mathrm{d} x\\
& \le \tau^{2\gamma} \dashint_{B_{2r}} \left|\bfV_{H^-_{2r}}(D\bfu)-\left(\bfV_{H^-_{2r}}(D\bfu)\right)_{2r}\right|^2 \,\mathrm{d} x\,. 
\end{split}\end{equation*}
This concludes the proof of \eqref{eq:degeneratedecayest}. 
\end{proof}

\subsection{Iteration in the nondegenerate regime} \label{sec:iterationnondeg}

In this section we set up the iteration scheme which proves the partial regularity 
of the weak solution $\bfu$ to the system \eqref{system1}. 
First we consider the nondegenerate case and, from Lemma~\ref{lem:nondegeneratedecay}, 
prove the following result.

\begin{lemma} 
\label{lemma:5:iteration-nondegenerate} 
 Let $B_{2R}(x_0)\Subset \Omega$ with $R\in(0,1/4)$ satisfy \eqref{eq:assless1} with $r=R$, and $0 < \beta  \le \alpha_3/4$, where $\alpha_3\in(0,1)$ 
is given by \eqref{alpha3}. 
There exist $\delta_5,\delta_6 > 0$ depending on $n$, $N$, 
$p$, $q$, $\nu$, $L$, $\alpha$, $[a]_{C^{0,\alpha}}$, $\beta_0$ and $\beta$  such that the following 
property holds: if 
\begin{equation} 
\label{eq:5:nondegenerate-decay-1}
      \dashint_{B_R(x_0)} \Big\vert \bfV_{H^{-}_{B_R(x_0)}}(D\bfu) - \left(\bfV_{H^-_{B_R(x_0)}}(D\bfu)\right)_{x_0,R} \Big\vert^2 \, \mathrm{d}x  
  \leq \delta_5 \dashint_{B_R(x_0)} \Big\vert \bfV_{H^-_{B_R(x_0)}}(D\bfu) \Big\vert^2 \, \mathrm{d}x
\end{equation}
and 
\begin{equation}                      
\label{eq:5:nondegenerate-decay-2}
R^{\frac{\alpha_3}{2}} 
  \leq \delta_6 \,, 
\end{equation} 
then we have
\begin{equation}
\label{eq:5:nondegenerate-holder}
\begin{aligned}
& \dashint_{B_r(x_0)} 
 \Big\vert \bfV_{H^-_{B_r(x_0)}}(D\bfu) - \left(\bfV_{H^-_{B_r(x_0)}}(D\bfu)\right)_{x_0,r} \Big\vert^2 \, \mathrm{d}x                              \\
& \leq c \left( \frac rR \right)^{2\beta}
       \dashint_{B_R(x_0)} \Big\vert \bfV_{H^-_{B_R(x_0)}}(D\bfu) - \left(\bfV_{H^-_{B_R(x_0)}}(D\bfu)\right)_{x_0,R} \Big\vert^2 \, \mathrm{d}x  
    +  c \, r^{2 \beta} 
       \dashint_{B_R(x_0)} \Big\vert \bfV_{H^-_{B_R(x_0)}}(D\bfu) \Big\vert^2 \, \mathrm{d}x
\end{aligned}
\end{equation}
for every $r \in (0,R)$.
\end{lemma} 

\begin{proof}
As usual, throughout the proof we omit the dependence on the point $x_0$, and write $H^\pm_{B_\rho(x_0)}=H^\pm_{\rho}$.

\smallskip 

\emph{Step~1. Choice of parameters.} 
Choose the parameters $\tau$ and $\varepsilon$ in Lemma~\ref{lem:nondegeneratedecay} 
as follows  
\begin{equation}
\label{eq:5:tau-epsilon}
  \tau
:= \min \left\{ 
           \left( \frac {1}{2c^\ast} \right)^{\frac{1}{1 - \beta}},
           \left( \frac {1}{16} \right)^{\frac {1}{1 - \beta}} 
        \right\} 
\qquad \text{and} \qquad 
   \varepsilon 
:= \frac {\tau^{n + 1 + \beta}}{2c^\ast},
\end{equation}
where the constant $c^\ast>0$ will be determined below.
This determines $\delta_1$ and $\delta_2$ in Lemma~\ref{lem:nondegeneratedecay}. 
We next choose $\delta_5$ and $\delta_6$ as follows: 
\begin{equation}
\label{eq:5:delta3-delta4}
   \delta_5
:= \min \left\{
           \delta_1,\frac {1}{8(1 + \tau^{-n})},
           \frac {(\sqrt2 - 1)^2(1 - \tau^{\beta})^2\tau^n}{2} 
        \right\} 
\quad \text{and} \quad
    \delta_6 
:= \min \left\{ \delta_2,\delta_5 \right\}.
\end{equation}


\emph{Step~2. Induction.}  
We prove by induction that the following inequalities hold 
\begin{subequations} 
\begin{align}
\label{eq:5:induction-1}
&      \dashint_{B_{\tau^kR}} 
       \vert \bfV_k(D\bfu) - (\bfV_k(D\bfu))_{\tau^kR} \vert^2 \, \mathrm{d}x 
  \leq \tau^{2\tilde{\beta}k}\delta_5 
       \dashint_{B_{\tau^kR}} \vert \bfV_k(D\bfu) \vert^2 \, \mathrm{d}x;                                  \\ 
\label{eq:5:induction-2} 
& \begin{aligned}
  &      \dashint_{B_{\tau^kR}} \vert \bfV_k(D\bfu) - (\bfV_k(D\bfu))_{\tau^kR} \vert ^2 \, \mathrm{d}x 
    \leq \tau^{(1 +\tilde{\beta)}k}
         \dashint_{B_R} \vert \bfV_0(D\bfu) - (\bfV_0(D\bfu))_R \vert^2 \, \mathrm{d}x                     \\
  & \hspace*{200pt} 
    +    \frac {1 - \tau^{(1-\tilde{\beta})k}}
                {1 - \tau^{1 - \tilde{\beta}}}(\tau^kR)^{2\tilde{\beta}}  
         \dashint_{B_R} 
         \hspace*{-5pt} 
         \vert \bfV_0(D\bfu) \vert^2 \, \mathrm{d}x;
  \end{aligned}                                                                        \\ 
\label{eq:5:induction-3}
&      \dashint_{B_{\tau^kR}} \vert \bfV_k(D\bfu) \vert^2 \, \mathrm{d}x
  \leq 2\dashint_{B_R} \vert \bfV_0(D\bfu) \vert^2 \, \mathrm{d}x 
\end{align}
\end{subequations}
for every $k \geq 0$, where
$\bfV_{k} := \bfV_{H^-_{\tau^kR}}$.
For convenience, in the sequel we shall write $\eqref{eq:5:induction-1}_k$, 
$\eqref{eq:5:induction-2}_k$ and $\eqref{eq:5:induction-3}_k$ to denote 
\eqref{eq:5:induction-1}, \eqref{eq:5:induction-2} and \eqref{eq:5:induction-3} 
for a specific value of $k$. Clearly, $\eqref{eq:5:induction-1}$, 
$\eqref{eq:5:induction-2}$ and $\eqref{eq:5:induction-3}$ hold for $k = 0$ by \eqref{eq:5:nondegenerate-decay-1}.

We next suppose that $\eqref{eq:5:induction-1}_h$, $\eqref{eq:5:induction-2}_h$ 
and $\eqref{eq:5:induction-3}_h$ hold for $h = 0,1,2,\ldots,k - 1$ for some $k \geq 1$ 
and then prove $\eqref{eq:5:induction-1}_k$, $\eqref{eq:5:induction-2}_k$ and 
$\eqref{eq:5:induction-3}_k$.
By \eqref{eq:5:delta3-delta4}, \eqref{eq:5:nondegenerate-decay-2} and $\eqref{eq:5:induction-1}_{k-1}$, we see that \eqref{smallness1} 
and \eqref{smallness2} hold for $r = \tau^{k - 1}R/2$. 
Hence, we can apply Lemma~\ref{lem:nondegeneratedecay} with $r = \tau^{k - 1}R/2$ and replacing $\tau$ by $2\tau$ 
to get 
\begin{align*}
&        \dashint_{B_{\tau^k R}} 
         \vert \bfV_k(D\bfu) - (\bfV_k(D\bfu))_{\tau^k R} \vert^2 \, \mathrm{d}x                                \\
& 
  \leq c^\ast \tau^2(1 + \varepsilon \tau^{-n - 2}) 
       \bigg( 
          \dashint_{B_{\tau^{k - 1}R}} \vert \bfV_{k-1}(D\bfu) - (\bfV_{k-1}(D\bfu))_{\tau^{k - 1}R} \vert^2 \, \mathrm{d}x                                
          + 
          (\tau^{k - 1}R)^{\frac{\alpha_3}{2}}  
          \dashint_{B_{\tau^{k - 1}R}} \vert \bfV_{k-1}(D\bfu) \vert ^2 \, \mathrm{d}x                                
       \bigg)
\end{align*} 
for some constant $c^*>0$.
Note that since $c^\ast\tau^{1 - \beta} \leq \frac{1}{2}$ by \eqref{eq:5:tau-epsilon},
$c^\ast \varepsilon\tau^{-\beta - n - 1} = \frac{1}{2}$ and
\begin{displaymath} 
     c^\ast \tau^2(1 + \varepsilon \tau^{-n - 2}) 
=    \tau^{1 + \beta}(c^\ast \tau^{1 - \beta} + c^\ast \varepsilon \tau^{-\beta - n - 1})
\leq \tau^{1 + \beta}.
\end{displaymath} 
Hence, recalling the facts that $\beta<\alpha_3/4$, $\delta_6 \leq \delta_5$ by \eqref{eq:5:delta3-delta4} and $\tau^{1 - \beta} \leq \frac{1}{16}$ 
by \eqref{eq:5:tau-epsilon},
and using $\eqref{eq:5:induction-1}_{k-1}$, \eqref{eq:5:nondegenerate-decay-2}, 
we see that 
\begin{equation} 
\label{eq:5:lem61pf2} 
\begin{aligned}
&      \dashint_{B_{\tau^kR}} \vert \bfV_{k}(D\bfu) - (\bfV_{k}(D\bfu))_{\tau^kR} \vert^2 \, \mathrm{d}x                                               \\
& \quad 
  \leq \tau^{1 + \beta} 
       \bigg( 
          \dashint_{B_{\tau^{k - 1}R}} 
          \vert \bfV_{k-1}(D\bfu) - (\bfV_{k-1}(D\bfu))_{\tau^{k - 1}R} \vert^2 \, \mathrm{d}x                                    
          +
          (\tau^{k - 1}R)^{\frac{\alpha_3}{2}}  
          \dashint_{B_{\tau^{k - 1}R}} \vert \bfV_{k-1}(D\bfu) \vert^2 \, \mathrm{d}x                               
       \bigg)                                                                          \\
& \quad 
  \leq \tau^{1 - \beta}\tau^{2\beta} 
       \bigg( 
          \tau^{2\beta(k - 1)} \delta_5
          \dashint_{B_{\tau^{k - 1}R}} \vert \bfV_{k-1}(D\bfu) \vert^2 \, \mathrm{d}x                                  
          + 
          \tau^{2\beta(k - 1)} \delta_6  
          \dashint_{B_{\tau^{k - 1}R}} \vert \bfV_{k-1}(D\bfu) \vert^2 \, \mathrm{d}x                                
       \bigg)                                                                          \\
& \quad 
  \leq \frac 18\tau^{2 \beta k} \delta_5 
       \dashint_{B_{\tau^{k - 1}R}} \vert \bfV_{k-1}(D\bfu) \vert^2 \, \mathrm{d}x                            . 
  \smash[b]{\vphantom{\left( \int_{B_R} \right)}}
\end{aligned}
\end{equation} 
On the other hand, by $\eqref{eq:5:induction-1}_{k-1}$ and the fact that 
$4(1 + \tau^{-n})\delta_5 \leq \frac{1}{2}$ by \eqref{eq:5:delta3-delta4}, we have 
\begin{displaymath} 
\begin{aligned}
        \dashint_{B_{\tau^{k - 1}R}} 
        \vert \bfV_{k-1}(D\bfu) \vert^2 \, \mathrm{d}x                               
& \leq 4\dashint_{B_{\tau^{k - 1}R}} 
        \vert \bfV_{k-1}(D\bfu) - (\bfV_{k-1}(D\bfu))_{\tau^{k - 1}R} \vert^2 \, \mathrm{d}x                                \\
& \qquad\quad
  +    4\vert (\bfV_{k-1}(D\bfu))_{\tau^{k - 1}R} - (\bfV_{k-1}(D\bfu))_{\tau^kR} \vert^2 
  +    4\dashint_{B_{\tau^kR}} \hspace*{-5pt} \vert \bfV_{k-1}(D\bfu) \vert^2 \, \mathrm{d}x                                                  \\
& \leq 4\left( 1 + \tau^{-n} \right) 
        \dashint_{B_{\tau^{k - 1}R}} 
        \vert \bfV_{k-1}(D\bfu) - (\bfV_{k-1}(D\bfu))_{\tau^{k - 1}R} \vert^2 \, \mathrm{d}x                                                             
  +    4\dashint_{B_{\tau^kR}} \vert \bfV_{k-1}(D\bfu) \vert^2 \, \mathrm{d}x                                                                 \\
& \leq 4\left( 1 + \tau^{-n} \right) \delta_5 
        \dashint_{B_{\tau^{k - 1}R}} \vert \bfV_{k-1}(D\bfu) \vert^2 \, \mathrm{d}x                                
       + 
       4\dashint_{B_{\tau^kR}} \vert \bfV_{k-1}(D\bfu) \vert^2 \, \mathrm{d}x                                                                 \\
& \leq \frac 12 
       \dashint_{B_{\tau^{k - 1}R}} \vert \bfV_{k-1}(D\bfu) \vert^2 \, \mathrm{d}x                                
       +
       4\dashint_{B_{\tau^kR}} \vert \bfV_{k-1}(D\bfu) \vert^2 \, \mathrm{d}x                                
\end{aligned}
\end{displaymath} 
which implies that
\begin{displaymath} 
      \dashint_{B_{\tau^{k - 1}R}} \vert \bfV_{k-1}(D\bfu) \vert^2 \, \mathrm{d}x                               
\leq 8\dashint_{B_{\tau^kR}} \vert \bfV_{k-1}(D\bfu) \vert^2 \, \mathrm{d}x                                \leq 8\dashint_{B_{\tau^kR}} \vert \bfV_{k}(D\bfu) \vert^2 \, \mathrm{d}x
\end{displaymath} 
since clearly $a^-_{\tau^{k-1} R} \leq a^-_{\tau^{k} R}$ and this gives $(H_{\tau^{k-1} R}^-)' \leq (H_{\tau^{k} R}^-)'$. Inserting this into \eqref{eq:5:lem61pf2}, we obtain $\eqref{eq:5:induction-1}_k$.

We next show that $\eqref{eq:5:induction-2}_k$ holds. 
From 
\eqref{eq:5:lem61pf2}, $\eqref{eq:5:induction-2}_{k-1}$ 
and $\eqref{eq:5:induction-3}_{k-1}$, we have 
\begin{displaymath} 
\begin{aligned}
&     \dashint_{B_{\tau^kR}} 
      \vert \bfV_{k}(D\bfu) - (\bfV_{k}(D\bfu))_{\tau^kR} \vert^2 \, \mathrm{d}x                                                                   \\
& 
  \leq \tau^{1 + \beta} 
       \bigg( 
          \dashint_{B_{\tau^{k - 1}R}} 
          \vert \bfV_{k-1}(D\bfu) - (\bfV_{k-1}(D\bfu))_{\tau^{k - 1}R} \vert^2 \, \mathrm{d}x                                    
          +
          (\tau^{k - 1}R)^{\frac{\alpha_3}{2}}  
          \dashint_{B_{\tau^{k - 1}R}} \vert \bfV_{k-1}(D\bfu) \vert^2 \, \mathrm{d}x                               
       \bigg)              \\                        
&
  \leq \tau^{1 + \beta} 
       \dashint_{B_{\tau^{k - 1}R}} \vert \bfV_{k-1}(D\bfu) - (\bfV_{k-1}(D\bfu))_{\tau^{k - 1}R} \vert^2 \, \mathrm{d}x                                
       +    
       \tau^{1-\beta}(\tau^kR)^{2\beta}  
       \dashint_{B_{\tau^{k - 1}R}} \vert \bfV_{k-1}(D\bfu) \vert^2 \, \mathrm{d}x                                                            \\
&  
  \leq \tau^{( 1 + \beta)k} 
       \dashint_{B_R} \vert \bfV_{0}(D\bfu) - (\bfV_{0}(D\bfu))_R \vert^2 \, \mathrm{d}x                                                              \\ 
&  
    \quad   + 
        \tau^{1 + \beta} \frac {1 - \tau^{(1 - \beta)(k - 1)}} 
              {1 - \tau^{1 - \beta}}(\tau^{k - 1}R)^{2\tilde{\beta}} 
       \dashint_{B_R} \vert \bfV_{0}(D\bfu) \vert^2 \, \mathrm{d}x                                     
       + 
       (\tau^{k}R)^{2\beta}  
       \dashint_{B_R} \vert \bfV_{0}(D\bfu) \vert^2 \, \mathrm{d}x                                                                         \\
&  
  =    \tau^{(1 + \beta)k} 
       \dashint_{B_R} \vert \bfV_{0}(D\bfu) - (\bfV_{0}(D\bfu))_R \vert^2 \, \mathrm{d}x                                
       + 
       \frac {1 - \tau^{(1 - \beta)k}} 
              {1 - \tau^{1 - \beta}}(\tau^kR)^{2\beta} 
       \dashint_{B_R} \vert \bfV_{0}(D\bfu) \vert^2 \, \mathrm{d}x                                
\end{aligned} 
\end{displaymath} 
which is $\eqref{eq:5:induction-2}_k$. 

Finally, by $\eqref{eq:5:induction-1}_h$ and $\eqref{eq:5:induction-3}_h$ with $h = 0,1,2,\dots,k - 1$ and the fact that 
$\tau^{-\frac n2}(2\delta_3)^{\frac 12} \frac {1}{1 - \tau^{\tilde{\beta}}} \leq \sqrt{2} - 1$ 
by \eqref{eq:5:delta3-delta4}, we obtain 

\begin{displaymath} 
\begin{aligned}
       \bigg( \dashint_{B_{\tau^kR}} \vert \bfV_{k}(D\bfu) \vert^2 \, \mathrm{d}x \bigg)^{\frac 12} 
& \leq \tau^{-\frac n2}
       \sum_{h = 0}^{k - 1} 
       \bigg(
          \dashint_{B_{\tau^hR}} \vert \bfV_{h+1}(D\bfu) - (\bfV_{h}(D\bfu))_{\tau^hR} \vert^2 \, \mathrm{d}x 
       \bigg)^{\frac 12}                                       
       +    
       \bigg( \dashint_{B_R} \vert \bfV_{0}(D\bfu) \vert^2 \, \mathrm{d}x \bigg)^{\frac 12}                 \\
& \leq \tau^{-\frac n2} \delta_5^{\frac 12} 
       \sum_{h = 0}^{k - 1} 
       \tau^{\beta h} 
       \bigg( \dashint_{B_{\tau^hR}} \vert \bfV_h(D\bfu) \vert^2 \, \mathrm{d}x \bigg)^{\frac 12} 
       +
       \bigg( \dashint_{B_R}  \vert \bfV_0(D\bfu) \vert^2 \, \mathrm{d}x \bigg)^{\frac 12}                 \\
& \leq \bigg( 
          \tau^{-\frac n2} (2\delta_5)^{\frac 12} \frac {1}{1-\tau^{\beta}} + 1 
       \bigg)
       \bigg(  \dashint_{B_R} \vert \bfV_0 (D \bfu) \vert^2 \, \mathrm{d}x \bigg)^{\frac 12}                \\
& \leq \bigg( 2\dashint_{B_R} \vert \bfV_0 (D \bfu) \vert^2 \, \mathrm{d}x \bigg)^{\frac 12} 
\end{aligned} 
\end{displaymath}
which implies $\eqref{eq:5:induction-3}_k$.

\smallskip 

\textit{Step~3. Decay estimates.}
Let $r \in (0,R)$. Then $\tau^{k + 1}R \leq r < \tau^kR$ for some $k \geq 0$. Therefore, 
by the same estimation as in \eqref{eq:VHcomparison2} we have
\begin{displaymath} 
\begin{aligned}
& \dashint_{B_r} \vert \bfV_{H^{-}_{r}}(D \bfu) - (\bfV_{H^{-}_r}(D \bfu))_r \vert^2 \, \mathrm{d}x 
\le \dashint_{B_r} \vert \bfV_{H^{-}_{r}}(D \bfu) - (\bfV_{k}(D \bfu))_{\tau^kR} \vert^2 \, \mathrm{d}x \\ 
&\le 2 \dashint_{B_{r}}\vert \bfV_{H^{-}_{r}}(D \bfu) - \bfV_{k}(D \bfu) \vert^2 \, \mathrm{d}x  + 2 \dashint_{B_{r}} \vert \bfV_{k}(D \bfu) - (\bfV_{k}(D \bfu))_{\tau^kR} \vert^2 \, \mathrm{d}x \\ 
&\le  c \tau^{-n}(\tau^k R)^{\alpha_2} \dashint_{B_{\tau^kR}} \vert \bfV_{k}(D \bfu) \vert^2 \, \mathrm{d}x  + c\tau^{-n} \dashint_{B_{\tau^kR}} \vert \bfV_{k}(D \bfu) - (\bfV_{k}(D \bfu))_{\tau^kR} \vert^2 \, \mathrm{d}x \\
&=: I+II \,.
\end{aligned}                         
\end{displaymath} 
For $I$, using \eqref{eq:5:induction-3}, 
we have 
$$
I \lesssim  \tau^{-n-\alpha_2}(\tau^{k+1} R)^{\frac{\alpha_2}{2}} \dashint_{B_{R}} \vert \bfV_{0}(D \bfu) \vert^2 \, \mathrm{d}x 
\lesssim  \tau^{-n-\alpha_2} r^{\frac{\alpha_2}{2}} \dashint_{B_{R}} \vert \bfV_{0}(D \bfu) \vert^2 \, \mathrm{d}x \,.
$$
For $II$, by \eqref{eq:5:induction-2}, we have 
\begin{displaymath} 
\begin{aligned}
II &
  \lesssim \tau^{-n}\tau^{(1 + \beta)k} 
       \dashint_{B_R} \vert \bfV_{0}(D\bfu) - (\bfV_0(D\bfu))_R \vert^2 \, \mathrm{d}x  
       +    
       \tau^{-n}\frac {1 - \tau^{(1 - \beta)k}} 
                       {1 - \tau^{1 - \beta}}(\tau^kR)^{2\beta} 
       \dashint_{B_R} \vert \bfV_{0}(D\bfu) \vert^2 \, \mathrm{d}x                                           \\
&  \lesssim  \tau^{-n-1-\beta} \left( \frac {r}{R} \right)^{2\beta}
       \dashint_{B_R} \vert \bfV_0(D\bfu) - (\bfV_{0}(D\bfu))_R \vert^2 \, \mathrm{d}x  
  +    \frac {\tau^{-n-1-2\beta}}{1 - \tau^{1 - \beta}}
       \left( \frac {r}{\tau} \right)^{2 \beta} 
       \dashint_{B_R} \vert \bfV_0(D\bfu) \vert^2 \, \mathrm{d}x\,.
\end{aligned} 
\end{displaymath} 
Consequently, recalling the definition \eqref{eq:5:tau-epsilon} of $\tau$, we obtain 
\eqref{eq:5:nondegenerate-holder}.
\end{proof}

\section{Proof of Theorem~\ref{thm:1:main-thm}} \label{sec:proofmainthm}

We are now in position to prove Theorem~\ref{thm:1:main-thm}. 

\begin{proof}[Proof of Theorem~\ref{thm:1:main-thm}]
Let $\gamma_0$ be the exponent of Lemma~\ref{lemma:2:G-harmonic-holder}, and fix $\gamma \in (0,\gamma_1)$, where
\begin{displaymath} 
   \beta 
:= \min \left\{ \frac{\gamma_0}{2}, \frac {\alpha_3}{4}\right\}\,. 
\end{displaymath} 
With this $\beta$, we find $\delta_5$ and $\delta_6$ in Lemma~\ref{lemma:5:iteration-nondegenerate}.
We also choose $\chi = \delta_5$  and $\gamma=\beta$ in Lemma~\ref{lemma:4:degenerate-decay}.
Consequently, $\delta_5$ and $\delta_6$ in Lemma~\ref{lemma:5:iteration-nondegenerate} 
and $\delta_3$, $\delta_4$ and $\tau$ in Lemma~\ref{lemma:4:degenerate-decay} 
are determined and depend only on the structure constants. 

Now, choose any point $x_1 \in \Omega$ satisfying
\begin{displaymath} 
\liminf_{r \to 0^+} \dashint_{B_r(x_1)} \Big\vert \bfV_{H^-_{B_r(x_1)}}(D\bfu) - (\bfV_{H^-_{B_r(x_1)}}(D\bfu))_{x_1,r} \Big\vert^2 \, \mathrm{d}x = 0 
\end{displaymath} 
and 
\begin{displaymath} 
   M 
:= \limsup_{r \to 0^+} \dashint_{B_r(x_1)} \Big\vert  \bfV_{H^-_{B_r(x_1)}}(D\bfu) \Big\vert^2 \, \mathrm{d}x 
<  +\infty\,. 
\end{displaymath} 
Fix $\Omega'\Subset\Omega$ such that $x_1\in \Omega'$. Note that there exists $r_0 \in (0,1/4)$ such that the inequalities in \eqref{eq:assless1} hold whenever $B_{2r}(x_0)\subset \Omega'$ and $r\in (0,r_0)$. Moreover, we can find $R_0\in (0,r_0/2)$ such that $B_{2R_0}(x_1)\subset \Omega'$,
and moreover 
\begin{equation}
\label{eq:5:mainthmpf-0}
R_0^{\frac{\alpha_3}{2}} \leq \min\left\{\frac {\delta_3}{4(M + 1)}, \delta_6,\delta_4\right\}\,, 
\end{equation}
\begin{equation*}
\dashint_{B_{R_0}(x_1)} \Big\vert  \bfV_{H^-_{B_{R_0}}(x_1)} (D\bfu) - ( \bfV_{H^-_{B_{R_0}(x_1)}}(D\bfu))_{x_1,R_0} \Big\vert^2 \, \mathrm{d}x 
\leq \frac {\delta_3}{4}
\ \mbox{ and } \ 
 \dashint_{B_{R_0}(x_1)} \Big\vert  \bfV_{H^-_{B_{R_0}(x_1)}}(D\bfu) \Big\vert^2 \, \mathrm{d}x 
\leq M + 1\,.
\end{equation*}
Therefore, by the continuity of the integrals above with respect to the translation 
of the domain of integration, there exists $R_1\in(0,R_0)$ such that for 
every $x_0 \in B_{R_1}(x_0)$ we have 
\begin{equation}
\label{eq:5:mainthmpf-1}
\dashint_{B_{R_0}(x_0)} \bigg \vert \bfV_{H^-_{B_{R_0}(x_0)}}(D\bfu) - (\bfV_{H^-_{B_{R_0}(x_0)}}(D\bfu))_{x_0,R_0} \bigg \vert^2 \, \mathrm{d}x \leq \frac {\delta_3}{2} 
\ \ \text{and}\ \ 
\dashint_{B_{R_0}(x_0)} \Big \vert \bfV_{H^-_{B_{R_0}(x_0)}}(D\bfu) \Big \vert^2 \, \mathrm{d}x \leq 2(M + 1)\,.
\end{equation}

Now, we fix an arbitrary point $x_0\in B_{R_1}(x_1)$, and write
$$
\bfV_{k}(\bfP) := \bfV_{H^-_{B_{\tau^k R_0}(x_0)}} (\bfP)\,, 
\qquad \bfP\in \R^{N\times n}\,.
$$
As usual,  throughout the remaining part we omit the dependence on the point $x_0$ and write $H_r^\pm:= H^\pm_{B_r(x_0)}$. We first suppose that 
\begin{equation} 
\label{eq:5:mainthmpf-2}
     \delta_5\dashint_{B_{\tau^kR_0}} \vert \bfV_{k}(D\bfu) \vert^2 \, \mathrm{d}x 
\leq \dashint_{B_{\tau^kR_0}} \vert \bfV_{k}(D\bfu) - (\bfV_{k}(D\bfu))_{\tau^{k}R_0} \vert^2 \, \mathrm{d}x 
\qquad 
\text{for every $k \geq 0$}\,.
\end{equation}
In view of \eqref{eq:5:mainthmpf-0} and \eqref{eq:5:mainthmpf-1},
 applying Lemma~\ref{lemma:4:degenerate-decay} inductively 
for $r=\tau^kR_0/2$, we have   
\begin{equation} 
\label{eq:5:mainthmpf-2bis} 
\begin{aligned}
       & \dashint_{B_{\tau^kR_0}} 
       \vert \bfV_{k}(D\bfu) - (\bfV_{k}(D\bfu))_{\tau^kR_0} \vert^2 \, \mathrm{d}x \\
& \,\,\,\, \leq \tau^{2\beta} 
       \dashint_{B_{\tau^{k -1}R_0}} 
       \vert \bfV_{k-1}(D\bfu) - (\bfV_{k-1}(D\bfu))_{\tau^{k - 1}R_0} \vert^2 \, \mathrm{d}x \\ 
& \,\,\,\, \leq \ldots \leq 
       \tau^{2k\beta} 
       \dashint_{B_{R_0}} 
       \vert \bfV_{0}(D\bfu) - (\bfV_{0}(D\bfu))_{R_0} \vert^2 \, \mathrm{d}x \leq  \tau^{2k \beta} \frac{\delta_3}{2} 
\end{aligned} 
\end{equation} 
holds for every $k \geq 0$. Hence, for $r \in (0,R_0)$ there exists $k \geq 0$ such that 
$\tau^{k + 1}R_0 \leq r < \tau^kR_0$ and so,
by arguing as in \eqref{eq:VHcomparison2} and  using \eqref{eq:5:mainthmpf-2bis},
\begin{displaymath} 
\begin{aligned}
& \dashint_{B_r} \vert \bfV_{H^{-}_{r}}(D \bfu) - (\bfV_{H^{-}_r}(D \bfu))_r \vert^2 \, \mathrm{d}x 
\le \dashint_{B_r} \vert \bfV_{H^{-}_{r}}(D \bfu) - (\bfV_{k}(D \bfu))_{\tau^kR_0} \vert^2 \, \mathrm{d}x                                    \\ 
&\le 2 \dashint_{B_{r}}|\bfV_{H^{-}_{r}}(D \bfu) - \bfV_{k}(D \bfu) \vert^2 \, \mathrm{d}x  + 2 \dashint_{B_{r}} \vert \bfV_{k}(D \bfu) - (\bfV_{k}(D \bfu))_{\tau^kR_0} \vert^2 \, \mathrm{d}x \\ 
&\le  c \tau^{-n}(\tau^k R_0)^{\alpha_2} \dashint_{B_{\tau^kR_0}} \vert \bfV_{k}(D \bfu) \vert^2 \, \mathrm{d}x  + c\tau^{-n} \dashint_{B_{\tau^kR_0}} \vert \bfV_{k}(D \bfu) - (\bfV_{k}(D \bfu))_{\tau^kR_0} \vert^2 \, \mathrm{d}x \\
&\le  c (\delta_5^{-1} + 1) \tau^{-n} \dashint_{B_{\tau^kR_0}} \vert \bfV_{k}(D \bfu) - (\bfV_{k}(D \bfu))_{\tau^kR_0} \vert^2 \, \mathrm{d}x \\
&\leq 
     c \delta_3(\delta_5^{-1} + 1)\tau^{-n}   \tau^{2k\beta} 
     \leq c \delta_3(\delta_5^{-1}+1)\tau^{-n}    \left(\frac{r}{\tau R_0}\right)^{2\beta} \,.
\end{aligned}
\end{displaymath}
Therefore, 
we have 
\begin{equation}
\label{eq:5:mainthmpf-3}
     \dashint_{B_r} \frac{\vert \bfV_{H^-_{r}}(D\bfu) - (\bfV_{H^-_r}(D\bfu))_{y,B_r(y)} \vert^2}{r^{2\beta}} \, \mathrm{d}x
\leq \frac {c\delta_3 (\delta_5^{-1}+1)}{\tau^{n + 2\beta} R_0^{2\beta}}\,.
\end{equation}

We next suppose that \eqref{eq:5:mainthmpf-2} does not hold. Then, there exists $k_0 \geq 0$ such that
\begin{equation}
\label{eq:5:mainthmpf-4}
     \delta_5
     \dashint_{B_{\tau^k R_0}} \vert \bfV_{k}(D\bfu) \vert^2 \, \mathrm{d}x\\
\leq \dashint_{B_{\tau^k R_0}} \vert \bfV_{k}(D\bfu) - (\bfV_{k}(D\bfu))_{\tau^{k}R_0} \vert^2 \, \mathrm{d}x 
\end{equation}
for every $k = 0,\dots,k_0 - 1$ (when $k_0=0$, \eqref{eq:5:mainthmpf-4} is meaningless) and 
\begin{equation} 
\label{eq:5:mainthmpf-5}
  \dashint_{B_{\tau^{k_0}R_0}} \vert \bfV_{k_0}(D\bfu) - (\bfV_{k_0}(D\bfu))_{\tau^{k_0}R_0} \vert^2 \, \mathrm{d}x
< \delta_3
  \dashint_{B_{\tau^{k_0} R_0}} \vert \bfV_{k_0}(D\bfu) \vert^2 \, \mathrm{d}x\,.
\end{equation}
If $k_0 = 0$, in view of Lemma~\ref{lemma:5:iteration-nondegenerate} with $R = R_0$ and  
of \eqref{eq:5:mainthmpf-1}, for every 
$r \in (0,R_0)$ we have 
\begin{align*}
       \dashint_{B_r} \vert \bfV_{H^-_{r}}(D\bfu) - (\bfV_{H^-_{r}}(D\bfu))_{r} \vert^2 \, \mathrm{d}x 
& \leq c  \left( \frac {r}{R_0} \right)^{2\beta}
       \dashint_{B_{R_0}} \vert \bfV_{0}(D\bfu) - (\bfV_{0}(D\bfu))_{R_0} \vert^2 \, \mathrm{d}x        + 
       c r^{2\beta} \dashint_{B_{R_0}} \vert \bfV_0(D\bfu) \vert^2 \, \mathrm{d}x                     \\
& \leq c  \delta_3\left( \frac {r}{R_0} \right)^{2\beta} 
       + 
       c  r^{2\beta}(M + 1)
\end{align*} 
and so
\begin{equation} 
\label{eq:5:mainthmpf-6}
\dashint_{B_r}\frac{\vert\bfV_{H^-_{r}}(D\bfu) - (\bfV_{H^-_{r}}(D\bfu))_{r}  \vert^2}{r^{2\beta}}\,\mathrm{d}x\leq c\left(\frac{\delta_3}{R_0^{2\beta}}+M+1\right)\,.
\end{equation}
It remains the case when \eqref{eq:5:mainthmpf-4} and \eqref{eq:5:mainthmpf-5} hold 
for some $k_0 \geq 1$. For $r \in [\tau^{k_0}R_0,R_0)$, we obtain \eqref{eq:5:mainthmpf-3} 
by the very same argument already used when \eqref{eq:5:mainthmpf-2} holds. 
On the other hand, if $r \in (0,\tau^{k_0} R_0)$, 
by Lemma~\ref{lemma:5:iteration-nondegenerate} with $R = \tau^{k_0}R_0$ and 
\eqref{eq:5:mainthmpf-3} with $r = \tau^{k_0}R_0$, we have  
\begin{align*}
&      \dashint_{B_r} \vert \bfV_{H^-_{r}}(D\bfu) - (\bfV_{H^-_{r}}(D\bfu))_{r} \vert^2 \, \mathrm{d}x \\
& \qquad\quad 
  \leq c  \left( \frac {r}{\tau^{k_0}R_0} \right)^{2\beta} 
       \dashint_{B_{\tau^{k_0}R_0}} \vert \bfV_{k_0}(D\bfu) - (\bfV_{k_0}(D\bfu))_{\tau^{k_0}R_0} \vert^2 \, \mathrm{d}x 
       + 
       c  r^{2\beta} 
       \dashint_{B_{\tau^{k_0}R_0}} \vert \bfV_{k_0}(D\bfu) \vert^2 \, \mathrm{d}x                           \\
& \qquad\quad 
  \leq c \frac{\delta_3}{2\tau^{n + 2\beta}}
       \left( \frac {r}{R_0} \right)^{2\beta} 
       + 
       c  r^{2\beta} 
       \dashint_{B_{\tau^{k_0}R_0}} \vert \bfV_{k_0}(D\bfu) \vert^2 \, \mathrm{d}x\,.
\end{align*}
Moreover,
by arguing as in \eqref{eq:VHcomparison2} and \eqref{eq:5:mainthmpf-2bis} for $k=k_0-1$ and using   \eqref{eq:5:mainthmpf-4}  for $k=k_0-1$,
\begin{align*}
    \dashint_{B_{\tau^{k_0}R_0}} \vert \bfV_{k_0}(D\bfu) \vert^2 \, \mathrm{d}x 
    &\le  2 \dashint_{B_{\tau^{k_0}R_0}} \vert \bfV_{k_0}(D\bfu)-\bfV_{k_0-1}(D\bfu) \vert^2 \, \mathrm{d}x + 2 \dashint_{B_{\tau^{k_0}R_0}} \vert \bfV_{k_0-1}(D\bfu) \vert^2 \, \mathrm{d}x\\
& \leq c \tau^{-n}  \dashint_{B_{\tau^{k_0-1}R_0}} \vert \bfV_{k_0-1}(D\bfu) \vert^2 \, \mathrm{d}x   \\
& \leq c\tau^{-n}  \delta^{-1}_5 
       \dashint_{B_{\tau^{k_0 - 1}R_0}} 
       \vert \bfV_{k_0-1}(D\bfu) - (\bfV_{k_0-1}(D\bfu))_{\tau^{k_0 - 1}R_0} \vert^2 \, \mathrm{d}x                        \\
& \leq c \tau^{-n} \delta^{-1}_5 \delta_3.
       \smash[b]{\vphantom{\int}}
\end{align*}
Therefore, for every $r \in (0,R_0)$ we have 
\begin{equation}
\label{eq:5:mainthmpf-7}
     \dashint_{B_r} \frac{\vert \bfV_{H^-_r}(D\bfu) - (\bfV_{H^-_{r}}(D\bfu))_{r} \vert^2}{r^{2\beta}} \, \mathrm{d}x
\leq \frac{c \delta_3}{\tau^{n + 2\beta}R_0^{2\beta}} 
     + 
     c \tau^{-n}  \delta^{-1}_5 \delta_3. 
\end{equation}
Consequently, by \eqref{eq:5:mainthmpf-3}, \eqref{eq:5:mainthmpf-6} and 
\eqref{eq:5:mainthmpf-7} we conclude that the inequality 
\begin{displaymath} 
     \dashint_{B_r(x_0)}\frac{\vert \bfV_{H^-_{B_r(x_0)}}(D\bfu) - (\bfV_{H^-_{B_r(x_0)}}(D\bfu))_{x_0,r} \vert^2}{r^{2\beta}}\,\mathrm{d}x
\leq C
\end{displaymath} 
holds for every ball $B_r(x_0)$ with $x_0 \in B_{R_1}(x_1)$ and for every $r \in (0,R_0)$. Moreover, since  we have from \eqref{eq:linkVphi} that
$$\begin{aligned}
\vert \bfV_{H^-_r}(\bfP_1) - \bfV_{H^-_{r}}(\bfP_2) \vert^2
& \sim
 (H^-_{r})_{|\bfP_2|} (|\bfP_1-\bfP_2|)\\
& \sim  \vert \bfV_{p}(\bfP_1) - \bfV_{p}(\bfP_2) \vert^2 + a^-_{r} \vert \bfV_{q}(\bfP_1) - \bfV_{q}(\bfP_2) \vert^2 \\
&\ge \vert \bfV_{p}(\bfP_1) - \bfV_{p}(\bfP_2) \vert^2\,,
\end{aligned}
$$
the previous inequality together with \eqref{eq:equivalencebis} implies
\begin{displaymath} 
     \dashint_{B_r(x_0)}\frac{\vert \bfV_{p}(D\bfu) - (\bfV_{p}(D\bfu))_{x_0,r} \vert^2}{r^{2\beta}}\,\mathrm{d}x
\leq C\,.
\end{displaymath} 
Hence $\bfV_p(D\bfu)\in C^{0,\beta}(B_{R_1}(x_1);\R^{N\times n})$, and this concludes the proof.
\end{proof}

\section*{Acknowledgments }

G. Scilla and B. Stroffolini are members of Gruppo Nazionale per l'Analisi Matematica, la Probabilit\`a e le loro Applicazioni (GNAMPA) of INdAM.
B. Stroffolini has been supported by the International  agreement between Sogang University and University of Naples Federico II. J. Ok was supported by the National Research Foundation of Korea by the Korean
Government (NRF-2022R1C1C1004523).

\subsection*{Conflict of interest} The authors  declare no conflict of interest.

\subsection*{Data availability} Data sharing is not applicable to this article as obviously no datasets were generated or
analyzed during the current study.



\end{document}